              \def\version{15 February 2019}	        	%

\documentclass[reqno,11pt]{amsart} 
\usepackage[utf8]{inputenc}
\usepackage{amsmath} 
\usepackage[mathscr]{eucal}
\usepackage{amssymb}
\usepackage{srcltx} 
\usepackage{dsfont}
\usepackage{hyperref}
\usepackage{color}
\usepackage{enumerate}
\usepackage{tikz}
\usepackage{cases}

\numberwithin{equation}{section}
 
\newcommand{\ProofEnde}{\hfill {$\square$}}
\newcommand{\ProofTeilEnde}{\hfill {$\triangle$}}
\newcommand{\ExampleEnde}{\hfill {$\diamond$}}
\def\emptyset{\varnothing} 

\def\a{\alpha}

\def\d{{\rm d}} 
\def\e{\varepsilon} 
 
\newfam\Bbbfam 
\font\tenBbb=msbm10 
\font\sevenBbb=msbm7 
\font\fiveBbb=msbm5 
\textfont\Bbbfam=\tenBbb 
\scriptfont\Bbbfam=\sevenBbb 
\scriptscriptfont\Bbbfam=\fiveBbb 
 
\def\mmax{m_{\max}}
 
\newcommand{\m}     {\mathfrak{m}} 
 
\newcommand{\R}     {\mathbb{R}} 
 
\newcommand{\N}     {\mathbb{N}}

\newcommand{\smfrac}[2]{\textstyle{\frac {#1}{#2}}}

\def\1{{\mathchoice {1\mskip-4mu\mathrm l}      
{1\mskip-4mu\mathrm l} 
{1\mskip-4.5mu\mathrm l} {1\mskip-5mu\mathrm l}}} 
 
\def\comment#1{} 
\newtheoremstyle{thm}{2ex}{2ex}{\itshape\rmfamily}{} 
{\bfseries\rmfamily}{}{1.7ex}{} 
 
\newtheoremstyle{rem}{1.3ex}{1.3ex}{\rmfamily}{} 
{\itshape\rmfamily}{}{1.5ex}{}

\renewcommand{\theequation}{\thesection.\arabic{equation}} 
 
\newtheorem{theorem}{Theorem}[section] 
\newtheorem{lemma}[theorem]{Lemma} 
\newtheorem{prop}[theorem] {Proposition} 
\newtheorem{cor}[theorem]  {Corollary}
\newtheorem{claim}[theorem] {Claim}

\theoremstyle{definition}
\newtheorem{example}[theorem] {Example}

\newtheorem{remark}[theorem]  {Remark}

\renewcommand{\d}{{\rm d}} 
 
\newcommand{\eps}{\varepsilon} 
\newcommand{\Leb}{{\rm Leb}}

\newcommand{\dist}{{\operatorname {dist}}} 
\newcommand{\diam}{{\operatorname {diam}}}

\newcommand{\SIR}{\mathrm{SIR}}

\def\ellmin{\ell_{\min}}
\def\ellmax{\ell_{\max}}

\def\kmax{k_{\max}}

\newcommand{\Mcal}   {{\mathcal M }} 
 
\newcommand{\Ocal}   {{\mathcal O }}

\newcommand{\Scal}   {{\mathcal S }}

\newcommand\numberthis{\addtocounter{equation}{1}\tag{\theequation}}
\renewcommand{\e}   {{\operatorname e }}

\definecolor{Red}{rgb}{1,0,0}

\begin{document} 
 
\title[Routeing properties in a Gibbsian model]{Routeing properties in a Gibbsian model \\\medskip for highly dense multihop networks}
\author[Wolfgang König and András Tóbiás]{}
\maketitle
\thispagestyle{empty}
\vspace{-0.5cm}

\centerline{{\sc Wolfgang K\"onig\footnote{TU Berlin, Straße des 17. Juni 136, 10623 Berlin, and WIAS Berlin, Mohrenstra{\ss}e 39, 10117 Berlin, {\tt koenig@wias-berlin.de}}} and {\sc András Tóbiás\footnote{Berlin Mathematical School, TU Berlin, Straße des 17. Juni 136, 10623 Berlin, {\tt tobias@math.tu-berlin.de}}}}
\renewcommand{\thefootnote}{}
\vspace{0.5cm}
\centerline{\textit{WIAS Berlin and TU Berlin, and TU Berlin}}

\bigskip

\centerline{\small(\version)} 
\vspace{.5cm} 
 
\begin{quote} 
{\small {\bf Abstract:}} We investigate a probabilistic model for routeing in a multihop ad-hoc communication network, where each user sends a message to the base station. Messages travel in hops via other users, used as relays. Their trajectories are chosen at random according to a Gibbs distribution, which favours trajectories with low interference, measured in terms of signal-to-interference ratio. This model was introduced in our earlier paper \cite{KT17}, where we expressed, in the limit of a high density of users, the typical distribution of the family of trajectories in terms of a law of large numbers. 

In the present work, we derive its qualitative properties. We analytically identify the emerging typical scenarios in three extreme regimes. We analyse the typical number of hops and the 
typical length of a hop, and the deviation of the trajectory from the straight line, \eqref{regime1} in the limit of a large communication area and large distances, and~\eqref{regime2} in the limit of a strong interference weight. In both regimes, the typical trajectory approaches a straight line quickly, in regime~\eqref{regime1} with equal hop lengths. Interestingly, in regime~\eqref{regime1}, the typical length of a hop diverges logarithmically in the distance of the transmitter to the base station. We further analyse~\eqref{regime3} local and global repulsive effects of a densely populated subarea on the trajectories. 

\end{quote}




\setcounter{tocdepth}{3}


\setcounter{section}{0}

\section{Introduction} \label{sec-Intro}
\noindent In this work, we continue our research \cite{KT17} on a spatial Gibbsian model for random message routeing in a multihop ad-hoc network with device-to-device (D2D) communication. In \cite{KT17} we prepared for an analysis of the qualitative properties of the model by deriving simplifying formulas that describe the situation in a densely populated area in the sense of a law of large numbers. In the present work, we carry out this analysis and describe a number of characteristic properties of the message trajectories. In particular, we are interested in the interplay between probabilistic properties like {\em entropy} and energetic properties like {\em interference} and {\em congestion} and how this interplay influences geometric characteristics like number and lengths of the hops or shapes of the trajectories. Our goal is to identify some rules of thumbs in the relationships between all these quantities in asymptotic regimes in which they become particularly pronounced, like large areas and long trajectories, strong 
influence of interference, or local 
regions with a 
particularly high population. While our previous paper used mainly probabilistic methods, the present paper entirely employs analytic tools.

\subsection{The main features of the model}\label{sec-mainfeatures}

Let us introduce our telecommunication model. The communication area $W$ is a bounded set in $\R^d$, and it has a unique base station at the origin $o$. Many users are distributed in $W$ according to some measure. Each user sends out a message to the base station along a random multihop trajectory that uses other users as relays and has at most $\kmax$ hops. (There is no mobility of the users (nodes) in our model.) We are interested in the joint distribution of all these random message trajectories, conditional on the locations of the users. 

Our main idea is to study a trajectory distribution that favours configurations with a high service quality from the viewpoint of interference. Under the distribution, all the message trajectories are stochastically independent. Each individual trajectory is distributed in the following way. 
{\em A priori}, it has a uniform distribution (i.e., it chooses first a hop number $k\in\{1,\dots,\kmax\}$ and then a $k$-hop trajectory, both uniformly at random), and there is an exponential weight term penalizing {\em interference}. This term is the sum of the reciprocals of the {\em signal-to-interference ratios (SIRs)} for each hop of the message trajectory. (In Section~\ref{sec-SIRpenaltydiscussion} we will point out that the penalty term is a certain approximation of the {\em bandwidth} used for the multihop transmission.) That is, the total penalty is given to the entire trajectory collection in terms of a probability weight. In the language of statistical mechanics, such a probability measure for collections of trajectories is called a {\em Gibbs distribution}. The highest probability is attached to those trajectory families that realize the best compromise between entropy (i.e., probability) and energy (i.e., interference); i.e., the Gibbs distribution respects the transmission properties of the entire system. Note that there is no strict SIR threshold, i.e., hops with bad SIR values are suppressed but not forbidden, similarly to the setting of \cite{FDTsT07}. In Section~\ref{sec-SIRthreshold}, we comment on a version of the model where hops with low SIR values are excluded.

This model is of snap-shot type without time dependence. Indeed, we assume that all messages are transmitted, relayed, and received at the same time. Further, all users act as transmitters but also as relays and can receive and forward messages, while the base station has only the function of a receiver. According to this, we define SIR in such a way that the interference for any hop of any message is determined by the spatial positions of all users (i.e., the starting points of all message trajectories), analogously to the model considered in \cite{HJKP15}. The independence of message trajectories under our Gibbsian trajectory distribution is a consequence of this choice, and it would not hold in time-dependent versions of the model. We comment on possible ways of introducing a time dimension in the model and their effect on the notion of SIR in Section~\ref{sec-timedependence}.


Summarizing, we consider an ad-hoc network with D2D communication in a bounded communication area with a large number of users and a single base station. Nevertheless, we note that we are not aware of a real-world telecommunication network that works according to the routeing policy that we use in this paper. One of our motivations is to explore the physical effect of the penalization of the joint probability of the random paths, which are {\em a priori} randomly picked with equal probability: Does the (soft) requirement of a good transmission quality force the trajectories to choose geometrically the shortest route? What hop lengths do they choose? We would like to understand the interplay between entropy and interference-energy and emerging effects.

The idea of an optimal trade-off between entropy and energy is most clearly realized in a certain limiting sense in \cite[Theorem 1.4]{KT17}, which will be the starting point of the present paper and will be summarized in Section~\ref{sec-beta=0}. There, we carried out the limit of a high density of users, and we derived a kind of law of large numbers for the ``typical'' trajectory distribution, i.e., the joint trajectory distribution that has the highest probability under the Gibbs measure. 
The optimal trajectory collection was obtained as the minimizer of a characteristic {\em variational formula}. Roughly speaking, the variational formula is of the form ``minimize the sum of entropy and energy among all admissible trajectory families'', see Section~\ref{sec-varformula}.

In fact, in \cite{KT17} we considered an extended version of the above model with another exponential weight term penalizing {\em congestion}. This term counts the ordered pairs of incoming hops arriving at the relays in the system. This is certainly an important characteristics of the quality of service, as too high an accumulation of many messages at relays results in a delay.  An important property of this term is that it introduces dependence between the trajectories of different messages, unlike the interference term. Hence, this model represents a situation with a centralized choice of all trajectories in the spirit of a common welfare, instead of selfish routeing optimization. In Section~\ref{sec-gametheory}, we give a game-theoretic discussion of the two weight terms in the exponent in the light of traffic theory; more precisely we ask under what circumstances the optimization of the sum of these two terms can be called selfish or non-selfish. In Section~\ref{sec-MCMC}, we also make a connection between
this optimization and our model from the viewpoint of stochastic algorithms. According to the results of \cite{M18}, realizing our Gibbsian system numerically using Monte Carlo Markov chain methods is on average much more effective than finding the optimum. This gives another motivation for our model. 

With the above definition of message trajectories, we only consider \emph{uplink communication}, i.e., users transmitting messages to the base station. The \emph{downlink}, i.e., the reversed direction, works very similarly. We believe that all the results of \cite{KT17} as well as the ones of the present paper have an analogue for the downlink with an analogous proof, and we refrain from presenting details. 

\subsection{Goals}\label{sec-goals}

Our goal in the present paper is to understand the global effects that are induced in the Gibbsian system exclusively by entropy and energy into geometric properties of the trajectory collection. 
As our model depends on various parameters (size and form of the communication area, density of users, choice of the interference term, strength of interference weighting, etc.), this can be done rigorously only in certain {\em limiting} regimes, namely:
\begin{enumerate}[(1)]
\item\label{regime1} large communication area and long distances (and large hop numbers), 
\item\label{regime2} strong interference penalization, and 
\item\label{regime3} high local density of users on a subset of the communication area.
\end{enumerate} 
 
We are interested in geometric properties such as the typical hop lengths, the average number of hops, and the typical shape of the trajectory.
In regimes~\eqref{regime1} and~\eqref{regime2}, we expect that the typical trajectories approach straight lines, and in~\eqref{regime1} there is an additional question about the typical length of a hop and the number of hops. Here, we would like to understand how the quality of service becomes bad in a large telecommunication area and how many and how large hops the messages would like to take if the constraint $\kmax$ on the maximum number of hops is dropped. 

However, the regime~\eqref{regime3} and our questions here are of a different nature. We would like to determine if the presence of a subarea with a particularly high population density has a significant (positive or negative) impact on the effective use of the relaying system: on the one hand, the trajectories have more available relays in such an area, but on the other hand, the interference achieves high values there. This is a trade-off between entropy and energy that we want to understand.

Let us point out that we are going to work on these questions only in the case where only interference is penalized, but not congestion. We decided this because the description of the minimizer(s) of the variational formula in \cite[Proposition 1.3]{KT17}  is enormously implicit and cumbersome in general, but reduces, if the congestion term is dropped, to relatively simple formulas that are amenable for analytical investigations \cite[Proposition 1.5]{KT17}. In particular, only in this setting we know that the minimizer is unique. We believe that the main qualitative properties persist to the case where also congestion is penalized, as this is purely combinatorial and not spatial. In this paper, the congestion term appears only in modelling discussions. Its formal definition is postponed until Section~\ref{sec-alsocongestion}.

\subsection{Our findings}\label{sec-results}

In regimes~\eqref{regime1}--\eqref{regime2}, we will see that the typical trajectory follows a straight line with exponential decay of probabilities of macroscopic deviations from this shape. Moreover, in regime~\eqref{regime1} we will also find simple formulas for the asymptotic number of hops and the average length of a hop, which turns out to be the same for each hop of the trajectory. One of our most striking findings is that, in regime~\eqref{regime1}, the typical hop length diverges as a power of the logarithm of the distance between the transmitter and the base station, and hence the typical number of hops is sub-linear in the distance. This effect seems to come from the facts that the total mass of the intensity measure of the communication area diverges and that, {\em a priori}, i.e., before switching on the interference weight, every message trajectory of a given length has the same weight, even very unreasonable ones that have long spatial detours, e.g., many loops. 

However, in regime~\eqref{regime3}, we encounter different effects. First we see the following global effect on the total number of relaying hops in the entire system: if the communication area is small (in the sense that all the interferences in the system do not vary much), then the total number of relaying hops vanishes exponentially fast in the diverging parameter of the dense population, regardless of the choice of the densely populated subset, as long as it has positive Lebesgue measure. In some cases, we also detect a local effect on the relaying hops if the densely populated subset is very small: we demonstrate that a certain neighbourhood of that subset is definitely unfavourable for relaying hops for practically all the other users. This is a very clear effect coming from the high interference of the densely populated area, which expels the trajectories away.

Some of our results are easy to guess, and the main value of our work is the explicit characterization of the quantities and the derivation of exponential bounds for deviations. We formulated our results in quite simple settings, by putting the communication area equal to a ball and the user density equal to the Lebesgue measure, but it is clear that they can be extended into various directions with respect to more complex shapes and/or user distributions.

Based on our explicit formulas, we also provide simulations in Section~\ref{sec-simulations}. They illustrate that most of the effects that we derived analytically in limiting settings, i.e., for large values of the parameters, already appear in a very pronounced way for quite moderate values of the parameters.

\subsection{Related literature}\label{sec-literature}
The quality of service in highly dense relay-augmented ad-hoc networks has received particular interest in the last years. A multihop network with users distributed according to a Poisson point process with diverging intensity was investigated in \cite{HJKP15}. Using large deviations methods, that paper derives the asymptotic behaviour of rare frustration events such as many users having an unlikely bad quality of service for an unusually long period of time. \cite{HJP16} also describes frustration probabilities in a network, where relays have a bounded capacity, and users become frustrated when their connection to a relay is refused because it is already occupied; see also \cite{HJ17}.

One difference between these works and the Gibbsian model of the present paper introduced in \cite{KT17} is that the latter one uses a notion of quality of service for the entire system rather than for single transmissions. In particular, trajectories with bad SIR are \emph{a priori} not excluded. There is a random mechanism for choosing the message trajectories of all users, given the user locations, and our results hold almost surely with respect to the point process of user locations in the high-density limit. For these results, users need not form a Poisson point process, and they can even be located deterministically \cite[Section 1.7.4]{KT17}. This is also a difference from \cite{HJKP15,HJP16,HJ17}, where user locations are not fixed and their randomness is (at least partially) responsible for unlikely frustration events.

For literature remarks on the notion and use of SIR, in particular for multiple hops, regarding the choice of a bounded path-loss function, and about the interference penalization term, see Section~\ref{sec-SIRdefdiscussion} later.

Gibbs sampling was used for various aspects of modelling telecommunication networks, e.g., in \cite{ChBK16} for optimal placement of contents in a cellular network, and in \cite{BC12} for power control and associating users to base stations. These Monte Carlo Markov chain methods are used to decrease some kind of cost in the system via a random mechanism, with no easily implementable deterministic methods being available. Our Gibbsian model also has this property if both interference and congestion are penalized. The recent master's thesis of Morgenstern \cite{M18} investigated the use of a Gibbs sampler or a Metropolis algorithm for an experimental realization of our Gibbsian system; see Section~\ref{sec-MCMC} for a summary. 


As for mathematical works about message routeing in interference limited multihop ad-hoc networks, let us mention the papers \cite{BBM11,IV15}. In these works, users are randomly selected as transmitters or receivers in each time slot, the success of transmissions is determined by an SIR constraint, and the main question is about the finiteness of the expected delay and the positivity of the information velocity. Since in the model of the present paper bad SIR values are penalized ``softly", i.e., we do not require that each hop of each message trajectory have a sufficiently large SIR value, further our model does not include a time dimension, our main objects of study are of a different nature. 

\subsection{Organization of this paper}\label{sec-organization}

In Section \ref{sec-beta=0}, we present our Gibbsian model and the results of \cite{KT17} that are relevant for the investigations of the current paper. 
%

Each of the following three sections is devoted to one of our three theoretical investigations, which form the core of this paper, i.e., the analysis of the large-distance limit~\eqref{regime1} in Section~\ref{sec-largespace}, the limit of strong interference penalization~\eqref{regime2} in Section~\ref{sec-gammatoinfty} and the limit of high local density of users~\eqref{regime3} in Section~\ref{sec-largea}. Each of these sections gives the question, the results, the proofs and a discussion in the respective setting.

Section~\ref{sec-conclusion} contains modelling discussions and conclusions. Here we discuss the notion of SIR and sketch some possible extensions of the model. Further, we provide motivations for our Gibbsian ansatz in the case when both interference and congestion are penalized.

Section \ref{sec-gametheory} discusses the relevance and properties of our Gibbsian model and the related optimization problem in the light of game-theoretic considerations in traffic theory. 

Finally, Section~\ref{sec-simulations} gives numerical plots and studies about qualitative properties of our model. 

\section{The Gibbsian model and its behaviour in the high-density limit}\label{sec-beta=0}
In this section, we recall the Gibbsian model of \cite{KT17} and its properties in the limit of high density of users. We present the model in Section~\ref{sec-modeldef}, describe its behaviour in the high-density limit in Section~\ref{sec-LDana} and comment on the notion of the typical trajectory sent out by a user in Section~\ref{sec-typicaltrajectory}. The main objects we will consider in this paper are defined in Sections~\ref{sec-LDana} and~\ref{sec-typicaltrajectory}, while the nomenclature and interpretation of these objects originate from the preceding Section~\ref{sec-modeldef}.

\subsection{The Gibbsian model}\label{sec-modeldef}

We introduce the model that we study in the present paper. This model was introduced in \cite[Section 1.2.4]{KT17}; it is a special case of the general model of \cite{KT17}. Here we only consider the case where only interference is penalized and congestion is not. For the definition of the model where also congestion is weighted, we refer the reader to Section~\ref{sec-alsocongestion}.

For any $n \in \mathbb N$ and for any measurable subset $V$ of $\mathbb R^n$, let $\mathcal{M}(V)$ denote the set of all finite nonnegative Borel measures on $V$. We write $[k]=\lbrace 1,\ldots,k \rbrace$ for $k \in \mathbb N$.  

We are working in $\R^d$ with $d \in \N$ fixed. 
Let $W\subset \mathbb R^d$ be compact, the area of the telecommunication system, containing the origin $o$ of $\R^d$.  
Let $\mu \in \mathcal{M}(W)$ be an absolutely continuous measure on $W$ with $\mu(W)>0$. For $\lambda>0$, we let $X^\lambda=(X_i)_{i \in I^\lambda}=(X_i)_{i=1}^{N(\lambda)}$ be a Poisson point process in $W$ with intensity measure $\lambda\mu$. We refer to the $X_i$ as to the users of the wireless network, thus $N(\lambda)=\# I^\lambda$ is the number of users in the network.

Now, we introduce message trajectories. Fix $X \in X^\lambda$ and $k \in [\kmax]$. A \emph{message trajectory} $t$ from $X$ to $o$ with $|t|=k$ hops is of the form $t=(t_0,\dots,t_k)$, where $t_0=X$ is the transmitter, $t_1,\ldots,t_{k-1} \in X^\lambda$ are the relays and $t_k=o$ is the receiver. Our modelling assumption is that each user $X_i$ submits exactly one message to $o$ along a trajectory $s^i$ (i.e., $s_0^i=X_i$). Further, we write $s = (s^i)_{i \in I^\lambda}$ for the configuration of all these trajectories. We denote by $\mathcal S_{\kmax}(X^\lambda)$ the set of all such trajectory configurations.

Next, we introduce interference penalization. We choose a \emph{path-loss function}, which describes the propagation of signal strength over distance. This is a monotone decreasing, continuous function $\ell\colon [0,\infty) \to (0,\infty)$. 
A typical choice is $\ell$ corresponding to ideal Hertzian propagation, i.e.\ $\ell(r)=\min \lbrace 1, r^{-\alpha} \rbrace$, for some $\alpha>0$ (see e.g. \cite[Section II.]{GT08}). The \emph{signal-to-interference ratio (SIR)} of a transmission from $X_i \in X^\lambda$ to $x \in W$ in the presence of the users in $X^\lambda$ is defined \cite{HJKP15} as
\[ \SIR(X_i,x,X^\lambda)= \frac{\ell(|X_i-x|)}{\frac{1}{\lambda} \sum_{j \in I^\lambda} \ell(|X_j-x|)}. \numberthis\label{SIR} \]
The sum in the denominator of the right-hand side of \eqref{SIR} is the \emph{interference}. In fact, according to conventional nomenclature, one should say ``total received power" instead of interference and ``signal-to-total received power ratio" instead of SIR. We discuss this in Section~\ref{sec-SIRdefdiscussion}, where we also comment on the factor $\smfrac{1}{\lambda}$ in the denominator of \eqref{SIR} and on the effect of the boundedness of the path-loss function. In Section~\ref{sec-timedependence} we point out how \eqref{SIR} would change if we introduced time dependence in our model.

Let us fix a parameter $\gamma>0$. Now, for a message trajectory $t$ from $X$ to $o$, we define 
\[ \Pi_{\gamma,\lambda}(t) = N(\lambda)^{1-|t|} \prod_{l=1}^{|t|}  \exp \Big( -\gamma \SIR(t_{l-1},t_l,X^\lambda)^{-1} \Big). \numberthis\label{Pinotation}
\]
Now, the central object studied in \cite{KT17} is the following Gibbs distribution on the set of configurations of trajectories. For $s=(s^i)_{i \in I^\lambda} \in \Scal_{\kmax}(X^\lambda)$ put
\[ \mathrm P^{\gamma}_{\lambda,X^\lambda}(s) = \frac{1}{Z^{\gamma}_\lambda(X^\lambda)} \prod_{i \in I^\lambda} \Pi_{\gamma,\lambda}(s^i). \numberthis\label{Gibbsdistribution} \]
This is the Gibbs distribution with a uniform and independent \emph{a priori} measure (see~\cite[Section 1.2.2]{KT17} for details), subject to an exponential weight with the sums of the reciprocals of the SIR values of all hops. Here
\[ Z^{\gamma}_{\lambda}(X^\lambda) =\sum_{r=(r_i)_{i \in I^\lambda} \in \Scal_{\kmax}(X^\lambda)} \prod_{i \in I^\lambda} \Pi_{\gamma,\lambda}(r^i) \numberthis\label{zed} \]
is the normalizing constant, which is referred to as \emph{partition function}. Note that $\mathrm P^{\gamma}_{\lambda, X^\lambda}(\cdot)$ is random and defined conditional on $X^\lambda$, and it is a probability measure on $\mathcal{S}_{\kmax}(X^\lambda)$.

\subsection{The limiting behaviour of the telecommunication system}\label{sec-LDana}

We study the above wireless communication system in the {\em high-density limit} $\lambda\to\infty$. 


Now we summarize the results of \cite{KT17} 
that are relevant for the present paper.
For $k\in\N$, elements of the product space $W^k=W^{\{0,1,\dots,k-1\}}$ are denoted as $(x_0,\ldots,x_{k-1})$. For $l=0,\ldots,k-1$, the $l$-th marginal of a measure $\nu_k \in \mathcal{M}(W^k)$ is denoted by $\pi_l\nu_k\in \mathcal{M}(W)$, i.e., $\pi_l \nu_k(A)=\nu_k(W^{\{0,\dots,l-1\}}\times A\times W^{\{l+1,\dots,k-1\}})$ for any Borel set $A$ of $W$. 

We assume that  the {\em empirical measure} of $X^\lambda$ normalized by $1/\lambda$, i.e., the measure
\begin{equation}\label{userempirical}
L_\lambda=\frac{1}{\lambda} \sum_{i \in I^\lambda} \delta_{X_i} ,
\end{equation}
converges to $\mu$ almost surely in the {high-density limit} $\lambda \to \infty$. (We write $\delta_x$ for the Dirac measure at $x$.) This condition is satisfied e.g.~if $\lambda \mapsto X^\lambda$ is increasing; for further details see \cite[Section 1.7.4]{KT17}. However, note that $L_\lambda$ is not normalized; its total mass converges towards $\mu(W)$.

For fixed $k \in [\kmax]$ and for a trajectory collection $s=(s^i)_{i\in I^\lambda} \in \mathcal{S}_{\kmax}(X^\lambda)$, we define the empirical measure of all the $k$-hop trajectories of $s$ as
\begin{equation}\label{Rdef}
R_{\lambda,k}(s)=\frac{1}{\lambda} \sum_{i \in I^\lambda\colon |s^i|=k} \delta_{(s_0^i,\ldots,s_{k-1}^i)}  \in \mathcal{M}(W^k).
\end{equation}
This is the main object behind the following analysis; it registers where the main bulk of the trajectories runs. Note that $R_{\lambda,k}$ is not normalized. Since each user sends out exactly one message, we have 
\[ \sum_{k=1}^{\kmax} \pi_0 R_{\lambda,k}(s)=L_\lambda. \numberthis\label{i-discrete0}\]
This assumption can be relaxed, see Section~\ref{sec-multiplemessages} for a discussion about this. Note that for $X_i \in X^\lambda$ and $y \in W$, we have
\[ \SIR(X_i,y,X^\lambda)=\frac{\int_{W} L_\lambda(\d z) \ell(|z-y|)}{\ell(|X_i-y|)}. \numberthis\label{SIRexpressed} \]

We denote by $S$ a random variable with distribution ${\rm P}_{\lambda,X^\lambda}^{\gamma}$. Since $L_\lambda \Rightarrow \mu$ as $\lambda \to \infty$ and \eqref{i-discrete0} holds for any $\lambda>0$, subsequential limits of $(R_{\lambda,k}(S))_{k \in [\kmax]}$ in the coordinatewise weak topology are easily seen to be of the form $\Sigma=(\nu_k)_{k \in [\kmax]}$ with $\nu_k \in \Mcal(W^k)$, satisfying
\[ \sum_{k=1}^{\kmax} \pi_0 \nu_k=\mu, \numberthis\label{i-beta=0} \]
cf.~\cite[Section 3.4]{KT17}. For such $\Sigma$, 
we define the following analogue of \eqref{SIRexpressed} with $L_\lambda$ replaced by its limit $\mu$:
 \[ g(x,y)=\frac{\int_W \mu(\d z) \ell(|z-y|)}{\ell(|x-y|)}. \numberthis\label{gdef} \]

The key result \cite[Proposition 1.5, parts (3), (4)]{KT17} about the limiting behaviour of the telecommunication system that we will use this paper is the following.

\begin{prop}[Law of large numbers for the empirical measures]\label{prop-variationbeta=0} Let $\gamma>0$ and $\kmax \in \N \setminus \{ 1 \}$. Then, almost surely with respect to $X^\lambda$, as $ \lambda\to\infty$, the distribution of $\Sigma_\lambda(S)=(R_{\lambda,k}(S))_{k\in[\kmax]}$ converges coordinatewise weakly to the collection of measures
$\Sigma=(\nu_k)_{k=1}^{\kmax}$, where
\[ \nu_k(\d x_0,\ldots,\d x_{k-1} ) = \mu(\d x_0) A(x_0) \prod_{l=1}^{k-1} \frac{\mu(\d x_l)}{\mu(W)} \e^{-\gamma \sum_{l=1}^k g(x_{l-1},x_l)}, \quad x_k=o,~k \in [\kmax],  \numberthis\label{nukminimizerbeta=0} \]
and the normalizing function $A$ is defined as
\[ \frac1{A(x_0)}=\sum_{k=1}^{\kmax} \int_{W^{k-1}} \prod_{l=1}^{k-1} \frac{\mu(\d x_l)}{\mu(W)} \e^{-\gamma \sum_{l=1}^k g(x_{l-1},x_l)}, \qquad x_0 \in W, \numberthis\label{Adefnew} \]
so that \eqref{i-beta=0} holds.
\end{prop}
Let us note that the case $\kmax=1$ is trivial because in this case, all messages are transmitted directly to the base station, and thus
$(R_{\lambda,k}(S))_{k \in [1]}=(L_\lambda)$ converges coordinatewise weakly to $\Sigma=(\nu_1)$, where $\nu_1=\mu$. 

In the limiting measure \eqref{nukminimizerbeta=0}, the starting points of the $k$-hop message trajectories, $k\in [\kmax]$, are chosen according to the measure $\mu(\d x_0) A(x_0)$ and the $l$th relays according to the measure $\mu(\d x_l)/\mu(W)$ for all $l \in [k-1]$, exponentially weighted by the limiting interference penalization term $\gamma\sum_{l=1}^k g(x_{l-1},x_l)$. 

In Section~\ref{sec-varformula} we will explain that the measures \eqref{nukminimizerbeta=0} form the unique minimizer of a characteristic variational formula; this property will not be used for deriving the main results of the present paper.
For further details about the limiting behaviour of the system, see~\cite[Sections~1.6, 1.7]{KT17}.

\subsection{Interpretation of the limiting trajectory distribution}\label{sec-typicaltrajectory}

The purpose of the present paper is to make further qualitative assertions about the ``typical'' trajectory from a given transmission site $x_0\in W$ to the origin, after having taken the limit $\lambda\to\infty$. A definition of the ``typical'' trajectory as a random variable is not immediate, due to the nature of this setting. In the present paper, we will focus on the probability measure on $\bigcup_{k \in [\kmax]} (\{ k \} \times W^{k-1})$ given by its Radon--Nikodym derivative
\begin{equation}\label{Tdef}
T_{x_0}(k,x_1,\dots,x_{k-1})= \frac{\nu_k(\d x_0,\d x_1,\dots,\d x_{k-1})}{\big(\sum_{k=1}^{\kmax} \pi_0 \nu_k(\d x_0) \big)\mu(\d x_1)\dots\mu(\d x_{k-1})} = \frac{\nu_k(\d x_0,\d x_1,\dots,\d x_{k-1})}{\mu(\d x_0)\mu(\d x_1)\dots\mu(\d x_{k-1})},
\end{equation}
with respect to $\sum_{k\in[\kmax]}(\delta_k\otimes \mu^{\otimes (k-1)})$. This function is the main object of our study in the present paper. We normalized $T_{x_0}$ in such a way that $\sum_{k\in[\kmax]}\int_{W^{k-1}}T_{x_0}(k,x_1,\ldots,x_{k-1})\mu(\d x_1) \ldots \mu(\d x_{k-1})=1$. According to Proposition~\ref{prop-variationbeta=0},
\begin{equation}\label{Tident}
T_{x_0}(k,x_1,\dots,x_{k-1})=A(x_0) \mu(W)^{-(k-1)}\prod_{l=1}^{k-1}\e^{-\gamma \sum_{l=1}^k g(x_{l-1},x_l)},
\end{equation}
where we recall \eqref{gdef}. We will use the convention that the 0th coordinate of $T_{x_0}$ is the one corresponding to $k$ and the $l$th is the one corresponding to $x_l$, for $l \in [k-1]$. This way, the marginal $\pi_0 T_{x_0}$ is a measure on $[\kmax]$.

We note that also the measure $M=\sum_{k=1}^{\kmax} \sum_{l=1}^{k-1} \pi_l \nu_k$ carries interesting information about the system. Indeed, in \cite[Section 1.3]{KT17} it was explained that, at a position $x \in W$, the typical number of incoming hops of a user at $x$ is Poisson distributed with parameter $M(\d x)/\mu(\d x)$, and the total mass $M(W)$ is the amount of relaying hops in the entire system, with the convention that it is zero if every message hops directly into $o$ without any relaying hop. Part of our analysis will also be devoted explicitly to $M$, see Section~\ref{sec-largea}.

\section{Large communication areas with large transmitter--receiver distances}\label{sec-largespace}

This section is devoted to the analysis of the highly dense telecommunication system described in Section~\ref{sec-LDana} in regime~\eqref{regime1}, i.e., in the limit of a large communication area coupled with a large distance of the user from the base station. In Section~\ref{sec-kopt}, we  present our main results and in Section~\ref{sec-Aproof} we prove them. Section~\ref{sec-koptdiscussion} includes discussions related to this regime. 

\subsection{The typical number, length, and direction of hops in a large-distance limit} \label{sec-kopt}

In this section, the main object of interest is the typical shape of the trajectory from a certain site to the origin, in particular the typical length of any of the hops, the number of hops, and the spatial progress of the trajectory, in particular whether or not it runs along the straight line or how strongly it deviates from it. We will answer these questions for the special choice that $W$ is a closed ball around the origin, $\mu$ is the Lebesgue measure on $W$, and the path-loss function $\ell$ corresponds to ideal Hertzian propagation so that $b=\int_{\R^d} \ell(|x|) \d x<\infty$, that is, $\ell(r)=\min\{ 1, r^{-\alpha} \}$ for some $\alpha>d$.

Furthermore, in order to obtain a transparent picture and to derive a neat result, we will have to assume that the starting site of our trajectory is far away from the origin. 
In such a setting, it is plausible to expect that, as the radius of the ball tends to infinity, a proportion of users that tends to one takes the same order of magnitude of number of hops. This also gives information about the typical length and direction of each hop in large but still compact communication areas.

We will see that this setting exhibits the interesting property that the typical number of hops diverges to infinity as the distance of the user $x_0$ from $o$ tends to infinity, however, in a sublinear way, more precisely, like the distance divided by a power of its logarithm. Second, using the asymptotics of the value of this largest summand in \eqref{Adefnew}, one can conclude about the typical length of the hops and about how much they deviate from the straight line between the transmitter and the receiver $o$. In our specific setting, we will be able to obtain precise and explicit asymptotics for all these quantities. 

We denote the radius of the communication area $W=\overline{B_r(o)}$ by $r$, and we recall that $\kmax$ is the maximal hop number. We consider the limit of large $r$ and large $\kmax$.  We consider one user placed at $x_0\in W$ with a distance from the origin $|x_0|=r_0$ being large, such that $r>r_0$, but $r\asymp r_0$. (We write ``$\asymp$'' if the quotient of the two sides stays bounded and bounded away from zero.) Then one can say that for large $r$, $x_0$ is a ``typical'' location of a user in $\overline{B_r(o)}$, chosen uniformly at random.

In our first result, Theorem~\ref{theorem-expectedhops}, we examine the ``typical'' number of hops of a trajectory from $x_0$ to $o$ as a random variable under the marginal distribution $\pi_0 T_{x_0}$ on $\N$. According to \eqref{Tident}, in the present setting, this is given by
\begin{equation}
 \pi_0 T_{x_0}(k)=A(x_0) a_k(x_0)\qquad \mbox{where }
a_k(x_0)=\int_{(B_r(o))^{k-1}} \prod_{l=1}^{k-1} \Big(\omega_d r^{-d} \,\e^{-\gamma g(x_{l-1},x_l)}\,\d x_{l}\Big), \qquad x_k=o, \numberthis\label{wishfulA}
\end{equation}
where $\omega_d$ is the volume of the unit ball in $\R^d$, and we recall that $g(x_{l-1},x_l)=\frac{\int_W \d y\, \ell(|y-x_l|)}{\ell(|x_{l-1}-x_l|)}$. It is \emph{a priori} not clear what the relation between $r$ and $\kmax$ in the limiting setting should be in order to obtain interesting assertions. Nevertheless, while $A$ depends on $\kmax$ via the identity $1/A(x_0)=\sum_{k=1}^{\kmax} a_k(x_0)$, we observe that the terms 
$a_k(x_0)$ can also be defined for $k > \kmax$ analogously to \eqref{wishfulA} for $k \leq \kmax$. Thus, it will be our first task to find the asymptotics of $a_k(x_0)$ without any reference to $\kmax$. We encounter a large deviation principle on a quite surprising scale.


\begin{theorem}[Large deviations for the hop number] \label{theorem-expectedhops}
Fix $t\in (0,\infty)$. Then, in the limit $r_0\to\infty$ with $r>r_0=|x_0| \asymp r$, for any choice of $r_0\mapsto k(r_0)\in\N$,
\begin{numcases}{\frac{1}{r_0 \log^{1-1/\alpha} r_0} \log a_{k(r_0)} (x_0) }
=-(dt + b \gamma t^{1-\alpha} )+o(1) & if~$\frac{k(r_0)}{r_0\log^{-1/\alpha} r_0}\to t$, \label{first-hopLDP} \\
\leq - b \gamma t^{1-\alpha}+o(1) & if~$\frac{k(r_0)}{r_0\log^{-1/\alpha} r_0}\leq t+o(1)$ \label{second-hopLDP} \\
\leq -dt+o(1) & if~$\frac{k(r_0)}{r_0\log^{-1/\alpha} r_0}\geq t+o(1)$. \label{third-hopLDP}
\end{numcases}
where we recall that $b=\int_{\R^d} \d y\, \ell(|y|)$.
\end{theorem}

The upper bound in \eqref{second-hopLDP} follows from the convexity of $1/\ell(|\cdot|)$ and a comparison between the functionals $(x,y) \mapsto g(x,y)$ and $(x,y) \mapsto  b / \ell(|x-y|)$. 

Note that Theorem~\ref{theorem-expectedhops} identifies the growth of $\log a_{k(r_0)}(x_0)$ on the scale $r_0 \log^{1-1/\alpha}r_0$ for $k(r_0)$ on the scale $r_0 \log^{-1/\alpha}r_0$; indeed, the second and third line rule out small and large values of $k(r_0)$ on  that scale, and the first line identifies the precise dependence on the prefactor. In more technical terms, $a_k(x_0)$ satisfies, with $k(r_0)\asymp r_0\log^{-1/\alpha}(r_0)$, a large deviation principle on the scale $r_0 \log^{1-1/\alpha} r_0$ with rate function $t\mapsto dt + b \gamma t^{1-\alpha}$. It is easily seen that this rate function has a unique minimizer
\begin{equation}\label{easyminimum}
t^*=\arg\min_{t \in (0,\infty)} \big(dt + b \gamma t^{1-\alpha}\big) = \Big( \frac{b \gamma(\alpha-1)}d\Big)^{1/\alpha}
\end{equation}
with minimum value 
\[ d t^*+b\gamma(t^*)^{1-\alpha}
=\frac{(b\gamma)^{1/\alpha}}{(\alpha-1) d}\Big[d+\big((\alpha-1)d\big)^{1/\alpha}\Big]. \]

As a consequence, we have the following law of large numbers.

\begin{cor} \label{cor-A}
In the limit $r_0\to\infty$ with $r>r_0=|x_0| \asymp r$, any maximizer $k^*(r_0)$ of $\N\ni k\mapsto a_k(x_0)$ satisfies 
\[ k^*(r_0)\sim t^*\frac{ r_0}{  \log^{1/\alpha} r_0}.
 \numberthis\label{koptlimit} \]
Further, if $\kmax\geq k^*(r_0)$ for at least one such maximizer for all sufficiently large $r_0>0$, then we have
\[ \frac{1}{r_0 \log^{1-1/\alpha} r_0} \log \frac{1}{A(x_0)} \to 
-( d t^*+b\gamma(t^*)^{1-\alpha}). \numberthis\label{Aexprate} \]
\end{cor}

If $\kmax$ is smaller than all the minimizers, then the asymptotics of $A(x_0)$ depend on those of $a_{\kmax}(x_0)$ rather than on $a_{k^*(r_0)}(r_0)$, and \eqref{Aexprate} has to be adapted accordingly. We note that \eqref{Aexprate} requires only a lower bound on $\kmax$, and in Corollary \ref{cor-A}, $\kmax$ could be equal to $+\infty$ for each $r_0$; see Section~\ref{sec-kmax=inftydiscussion} for a discussion about allowing arbitrary many hops in our model. \eqref{Aexprate} says that the asymptotic logarithmic behaviour of $1/A(x_0)$ on scale $r_0 \log^{1-1/\alpha} r_0$ coincides with the one of the single maximal summand $a_{k^*(r_0)}(r_0)$. Formulated in terms of the marginal distribution $\pi_0 T_{x_0}$ of $T_{x_0}$ on the length $k$ of the path from $x_0$ to $o$, since the behaviour of the Lebesgue measure restricted to $\overline{B_r(o)}$ is subexponential in $r_0$ in the large-distance limit that we are considering, we have that
$$
\pi_0 T_{x_0}\Big([t^*-\eps,t^*+\eps]^{\rm c}\frac{r_0}{\log^{1/\alpha}(r_0)}\Big)
$$
tends to zero exponentially fast on the scale $r_0 \log^{1-1/\alpha} r_0$ for all $\eps>0$. In Section~\ref{sec-largespacediscussion} we give an explanation of how these scales arise. 

In the proof of the lower bound of \eqref{koptlimit}, considering a uniform hop length distribution was sufficient, i.e., $t^* r_0/\log^{1/\alpha} r_0$ hops along the same straight line directed from $x_0$ to $o$ with length $r_0/k(r_0) \sim \frac 1{t^*} \log^{1/\alpha} r_0$ each.  We now show, again in terms of a large deviations estimate on the scale $r_0 \log^{1-1/\alpha} r_0$, that macroscopic deviations from this optimal hop length on the scale $\log^{1/\alpha} r_0$ have extremely small probability. For a finite set $A$, we write $\# A$ for the cardinality of $A$. 

\begin{prop} \label{cor-straighthops}
For $\varepsilon,\delta>0$ and $k\in\N$, let
\begin{equation}\label{depsilondelta}
\begin{aligned}
D_{\varepsilon,\delta}(k,x_0) &= \Big\{ (x_1,\ldots,x_{k-1}) \in {B_r(o)}^{k-1} \colon \exists I \subseteq [k-1] \colon \# I \geq \delta k,\\
& \qquad \frac{1}{\# I} \sum_{l \in I} \frac{|x_{l-1}-x_{l}|-\big||x_{l-1}|-|x_{l}|\big|}{\log^{1/\alpha} r_0} > \varepsilon \Big\},\qquad x_{k}=o.
\end{aligned}
\end{equation}
Then, in the limit $r_0\to\infty$ with $r>r_0=|x_0| \asymp r$, for $k(r_0) \sim t^* r_0/\log^{1/\alpha} r_0$,
\begin{equation}\label{r0expdecay}
\limsup \frac1{r_0 \log^{1-1/\alpha} r_0}\log T_{x_0}\big(k(r_0),D_{\varepsilon,\delta}(k(r_0),x_0)\big)< 0.
\end{equation}
\end{prop}

In words, the probability that there are $\asymp k(r_0)$ hops $x_l-x_{l-1}$, $l\in I$ for some index set $I$, in the trajectory of relays such that their average hop length $ \frac1{\# I} \sum_{i=1}^{\# I} |x_{l_i}-x_{l_i-1}|$ deviates from  the optimal hop length $\frac 1{t^*} \log^{1/\alpha} r_0 \approx r_0/k^*(r_0)$ on that scale, decays exponentially fast to zero on the scale $r_0 \log^{1-1/\alpha} r_0$. (In the denominator of the summands in \eqref{depsilondelta}, we have removed the factor $\smfrac{1}{t^*}$ in order to simplify notation.) 

We presented the results of this section for the path-loss functions $\ell$ of the form $\ell(r)=\min \{ 1, r^{-\alpha} \}$, $\alpha>d$, which makes the notation in the proofs less heavy. However, these assertions require only two properties of $\ell$: the integrability of $\ell(|\cdot|)$ over $\R^d$ and the convexity of $1/\ell(|\cdot|)$, see Section~\ref{sec-pathlossdiscussion}.

The proofs of Theorem \ref{theorem-expectedhops}, Corollary \ref{cor-A} and Proposition \ref{cor-straighthops} are carried out in Sections \ref{sec-Aproofmain}, \ref{sec-ProofCor} and \ref{sec-straighthopproof}, respectively. A discussion about these results and their proofs can be found in Section \ref{sec-largespacediscussion}.


Certainly, our results of this section hold for much more general communication areas $W$, not only for balls. Essential for our approach is only that a -- in every space dimension diverging -- neighbourhood of the straight line between $x_0$ and $o$ is contained in $W$ in the limit considered. The parameter $d$ appearing in the rate function goes back to our assumption that the volume of $W$ grows like the $d$-th power of $r$; however, other powers than $d$ in $[1,d]$ are also possible by putting other geometric assumptions on $W$.

\subsection{Proof of the results of Section \ref{sec-kopt}} \label{sec-Aproof} 
All the three results of Section~\ref{sec-kopt} tell about the limit $r_0\to\infty$ with $r>r_0 \asymp r$, where $x_0 \in W=\overline{B_{r}(o)}$ has Euclidean norm $|x_0|=r_0$. Throughout this section, we will use the notation $\lim_{r,r_0}$ for this limit and refer to it as ``our limit''.

\subsubsection{Proof of Theorem \ref{theorem-expectedhops}} \label{sec-Aproofmain}

Our strategy for proving the three assertions \eqref{first-hopLDP}, \eqref{second-hopLDP} and \eqref{third-hopLDP} is the following. First we verify the lower bound in \eqref{first-hopLDP}. Then we prove \eqref{second-hopLDP} and afterwards \eqref{third-hopLDP}, and we combine these two proofs in order to conclude the upper bound in \eqref{first-hopLDP}.

\noindent \emph{Proof of~\eqref{first-hopLDP}, lower bound.} 
Let us first consider $k(r_0)$ satisfying just $k(r_0)=o( r_0)$. We obtain a lower bound for $a_{k(r_0)}(x_0)$ defined in \eqref{wishfulA} by restricting the $x_l$-integral to the ball with radius one around $\frac{k(r_0)-l}{k(r_0)}x_0$ for $l=1,\ldots,k(r_0)-1$. Then, eventually, $1 \leq 
|x_{l-1}-x_l|\leq|x_0|/k(r_0) + 2$ for $l=1,\ldots,k(r_0)$. Note that $g(x_{l-1},x_l)\leq b/\ell(|x_{l-1}-x_l|)=b |x_{l-1}-x_l|^{\alpha}$, where we recall that $b=\int_{\R^d} \d y\, \ell(|y|)$. Hence, for any $\eps\in(0,1)$, eventually,  
\[ 
g(x_{l-1},x_l)\leq b|x_{l-1}-x_l|^\alpha\leq (1+\eps)br_0^\alpha /k(r_0)^\alpha,\qquad l=1,\ldots,k(r_0), 
\] 
where in the first step we used that $\smfrac{r_0}{k(r_0)}=\smfrac{|x_0|}{k(r_0)}$ tends to infinity in our limit. This gives
$$
a_{k(r_0)} (x_0) \geq  (\omega_d r^d)^{-k(r_0)+1}  \e^{-\gamma b k(r_0) (1+\varepsilon)(r_0/k(r_0))^\alpha} \geq \e^{-(d+\eps)k(r_0)\log r_0 -\gamma b k(r_0) (1+\varepsilon)(r_0/k(r_0))^\alpha},
$$
where the second inequality holds eventually, since $r_0 \asymp r$. Now an elementary optimization on $k(r_0)$ shows that $k(r_0)\asymp r_0\log^{-1/\alpha}r_0$ is the relevant scale. Then, in the particular case that $k(r_0)\sim t r_0\log^{-1/\alpha}r_0$ for some $t\in(0,\infty)$, carrying out the limit and making $\eps\downarrow0$ afterwards, we have
\[ \liminf_{r, r_0} \frac{1}{r_0 \log^{1-1/\alpha} r_0} \log a_{k(r_0)} (x_0 ) \geq -\Big( dt+\gamma b t^{1-\alpha} \Big), \]
which is the lower bound in \eqref{first-hopLDP}. \ProofTeilEnde

\noindent\emph{Proof of~\eqref{second-hopLDP}}. This proof uses that $1/\ell$ is convex and that the numerator $\int_W \ell(|y-x_l|)\,\d y$ can be well approximated by $b$ for sufficiently many $l$. These arguments lead to the following lemma. 
\begin{lemma}\label{lemma-gapproxbound}
Let $\eps>0$. If $k(r_0)\leq \frac 12 r_0$ for all $r_0$ sufficiently large, then eventually in our limit,
\[\numberthis\label{gapproxbound}
\sum_{l=1}^{k(r_0)} g(x_{l-1},x_l) \geq (1-\varepsilon)^\alpha (b-\eps) k(r_0)^{1-\alpha} r_0^\alpha \]
holds simultaneously for all $x_0,\ldots,x_{k(r_0)-1} \in B_r(o)$ with $|x_0|=r_0$ and $x_{k(r_0)}=o$.
\end{lemma}
\begin{proof} Let us now define an auxiliary function $s\colon (0,\infty) \to(0,\infty)$ such that $ r-s(r)\to\infty$ and $0<r-s(r)=o(r)$ in our limit. Fix $\eps\in(0,\frac 14)$. The idea is to pick $r$ sufficiently large so that 
\[\numberthis\label{uniformity}
\int_{\overline{B_r(o)}} \ell(|y-x|) \,\d y \geq b-\eps,\qquad \forall x\in \overline{B_{s(r)}(o)}.
\]
Let us assume that we are given a trajectory $(x_0,x_1,\ldots,x_{k(r_0)-1},x_{k(r_0)})$ with $k(r_0) \leq \frac 12 r_0$, $|x_0|=r_0$ and $x_{k(r_0)}=o$. Let us define the index of the last hop outside $B_{s(r)}(o)$:
\[ 
K(r_0,r) = \max \lbrace l \in \lbrace 0,1,\ldots,k(r_0)-1 \rbrace \colon |x_l| \geq s(r) \rbrace,
\numberthis\label{k0def} 
\]
which we want to understand as $K(r_0,r) =0$ if there is no such hop. Let $r_0>0$ be sufficiently large so that $s(r)>(1-\varepsilon)  r$ and \eqref{uniformity} holds. Then we have
\begin{align}
\sum_{l=1}^{k(r_0)} g(x_{l-1},x_l) &\geq \sum_{l=K(r_0,r)+1}^{k(r_0)} g(x_{l-1},x_l)  \geq \sum_{l=K(r_0,r)+1}^{k(r_0)} \frac{b-\varepsilon}{\ell(|x_{l-1}-x_l|)} \label{buniform} \\ 
& \geq \frac{(b-\varepsilon) (k(r_0)-K(r_0,r)) }{\ell \Big( \frac{1}{k(r_0)-K(r_0,r)} \sum_{l=K(r_0,r)+1}^{k(r_0)} |x_{l-1}-x_l| \Big)} \label{Jensenagain} \\ 
& \geq \frac{(b-\varepsilon) (k(r_0)-K(r_0,r)) }{\ell \Big( \frac{1}{k(r_0)-K(r_0,r)} (1-\varepsilon) r_0 \Big)} \label{distancelarge} \\ 
& \geq (1-\varepsilon)^{\alpha} (b-\varepsilon) (k(r_0)-K(r_0,r))^{1-\alpha} r_0^\alpha \geq (1-\varepsilon)^{\alpha} (b-\varepsilon) (k(r_0))^{1-\alpha} r_0^\alpha. \label{lastline}
\end{align}
In \eqref{buniform} we used the fact that $x_{K(r_0,r)}, \ldots, x_{k(r_0)-1}, x_{k(r_0)}$ lie in  $B_{s(r)}(o)$ and therefore \eqref{uniformity} can be applied for the numerator of each $g(x_{l-1},x_l)$ with $l>K(r_0,r)$. Next, \eqref{Jensenagain} is an application of Jensen's inequality for $1/\ell(|\cdot|)$, and \eqref{distancelarge} uses the following fact. Either $K(r_0,r)=0$, in which case 
\[ 
\sum_{l=K(r_0,r)+1}^{k(r_0)} |x_{l-1}-x_l|\geq \sum_{l=K(r_0,r)+1}^{k(r_0)} (|x_{l-1}|-|x_l| )\geq |x_0|=r_0 \geq  2 k(r_0) \geq  k(r_0)-K(r_0,r), 
\numberthis\label{nothingaway}
\]
or $K(r_0,r)>0$, and thus
\[ 
\begin{aligned}
\sum_{l=K(r_0,r)+1}^{k(r_0)} |x_{l-1}-x_l|&\geq \sum_{l=K(r_0,r)+1}^{k(r_0)} (|x_{l-1}|-|x_l|) \geq s(r) \\
&\geq (1-\varepsilon) r > (1-\varepsilon) r_0 \geq k(r_0) \geq k(r_0) - K(r_0,r). 
\end{aligned}
\numberthis\label{somethingaway}
\]
In both cases, the argument in $\ell(\cdot)$ is $\geq 1$, and we can write the term in terms of the $\alpha$-norm and the first step in \eqref{lastline} also follows. Hence, we have derived \eqref{gapproxbound}. \end{proof} 
By Lemma~\ref{lemma-gapproxbound}, for any $\eps>0$, we have eventually in our limit, under the assumptions of the lemma
\begin{align*} 
& \int_{\overline{B_r(o)}^{k(r_0)-1}}\Big( \prod_{l=1}^{{k(r_0)}-1} \frac{\d x_l}{\Leb(\overline{B_r(o)})}\Big)\, \e^{-\gamma \sum_{l=1}^{k(r_0)} g(x_{l-1},x_l)}  \\ \leq & \int_{\overline{B_r(o)}^{k(r_0)-1}}\Big( \prod_{l=1}^{{k(r_0)}-1} \frac{\d x_l}{\Leb(\overline{B_r(o)})}\Big)\, \e^{-\gamma (1-\eps)^{\alpha} (b-\eps) {k(r_0)}^{1-\alpha} r_0^\alpha} 
= \e^{-\gamma (1-\eps)^{\alpha} (b-\eps) {k(r_0)}^{1-\alpha} r_0^\alpha}  .
\end{align*}
Now, let $t>0$ and $r_0 \mapsto k(r_0)$ such that $k(r_0) \leq (t+o(1))r_0/\log^{1/\alpha} r_0$ (in particular $k(r_0) \leq \frac 12 r_0$ eventually). Then,
\begin{equation}\label{finalrate} 
\begin{aligned} 
& \limsup_{r, r_0} \frac{1}{r_0 \log^{1-1/\alpha} r_0} \log a_{k(r_0)}(x_0) 
\leq \limsup_{r, r_0} \frac{-(b-\eps)(1-\eps)^{\alpha} \gamma k(r_0)^{1-\alpha} r_0^\alpha}{r_0 \log^{1-1/\alpha} r_0}\\
&= \limsup_{r,r_0} - (b-\eps)(1-\eps)^{\alpha} \gamma \Big( \frac{k(r_0) \log^{1/\alpha}(r_0)}{r_0}\Big)^{1-\alpha} \leq - (b-\eps)(1-\eps)^{\alpha} \gamma t^{1-\alpha}
\end{aligned} 
\end{equation}
for all $\eps>0$, and thus
\[ \limsup_{r, r_0} \frac{1}{r_0 \log^{1-1/\alpha} r_0} \log a_{k(r_0)}(x_0)  \leq - b \gamma t^{1-\alpha}, \]
which is \eqref{second-hopLDP}.
\ProofTeilEnde

\noindent\emph{Proof of~\eqref{third-hopLDP}}. Note that for any $x \in \overline{B_r(o)}$, we have
\[ \int_{B_r(o)} \ell(|y-r e_1|) \, \d y\leq \int_{B_r(o)} \ell(|y-x|) \, \d y, \]
where $e_1=(1,0,\ldots,0)$ is the first unit vector of $\R^d$. 

Let us introduce the quantity $b_0 = \lim_{r \to \infty} \int_{B_r(o)} \ell(|y-r e_1|)\, \d y=\sup_{r\in(0,\infty)}\int_{B_r(o)} \ell(|y-r e_1|)\, \d y\in(0,b)$.
Now, for any $k\colon (0,\infty) \to \mathbb N$, in our limit,
\begin{equation}\label{fullintegral}
\begin{aligned}
   \Leb(B_r(o))^{-(k(r_0)-1)}  a_{k(r_0)}(x_0)&=  \int_{(B_r(o))^{k(r_0)-1}} \Big(\prod_{l=1}^{k(r_0)-1}\d x_l\Big)\,\e^{-\gamma \sum_{l=1}^{k(r_0)} \frac{\int_{B_{r}(o)} \ell(|y-x_l|) \,\d y}{\ell(|x_{l-1}-x_l|)}} \\
    &\leq \int_{(\R^d)^{k(r_0)-1}}\prod_{l=1}^{k(r_0)-1}\Big(\d x_l\,\e^{-\gamma \frac{b_0-o(1)}{\ell(|x_{l-1}-x_l|)}}\Big) \\
     & \leq   \Big(\int_{\R^d} \e^{-\gamma \frac{b_0-o(1)}{\ell(|y|)}} \,\d y \Big) ^{k(r_0)-1} =O(1)^{k(r_0)} = \exp\big(o(k(r_0) \log r_0)\big),
\end{aligned}
\end{equation}
where the first step in the last line follows from an elementary substitution and a reversion of the order of integration. Now, recall that in our limit $r \asymp r_0$. If $t>0$ and $k(r_0) \geq (t + o(1)) r_0/\log^{1/\alpha} r_0$, we have that
\[  \Leb(B_r(o))^{-(k(r_0)-1)} = \exp(-(dt+o(1)) k(r_0) \log r_0)) \leq \exp\big( -(dt+o(1)) r_0 \log^{1-1/\alpha} r_0 \big).  \] 
This implies \eqref{third-hopLDP}. \ProofTeilEnde

\noindent\emph{Proof of \eqref{first-hopLDP}, upper bound.} We combine our arguments from the proofs of the upper bounds in \eqref{second-hopLDP} and \eqref{third-hopLDP} in order to obtain the upper bound in \eqref{first-hopLDP}. Indeed, for $t>0$ and $k(r_0) \sim t r_0 / \log^{1/\alpha} r_0$ and $\varepsilon>0$, let us write $g(x_{l-1},x_l)=\varepsilon g(x_{l-1},x_l)+(1-\varepsilon) g(x_{l-1},x_l)$, estimate the first term like in \eqref{fullintegral} and the second term with the help of \eqref{gapproxbound}. This gives eventually
\begin{align*}  a_{k(r_0)}(x_0) & 
\leq  \int_{W^{{k(r_0)}-1}}\Big( \prod_{l=1}^{{k(r_0)}-1} \frac{\d x_l}{\Leb(B_r(o))}\Big)\, \e^{-\varepsilon \gamma \sum_{l=1}^{k(r_0)} \frac{b_0-o(1)}{\ell(|x_{l-1}-x_l|)}   } \e^{- (1- \varepsilon)(1-\varepsilon)^{\alpha} (b-\varepsilon)\gamma t^{1-\alpha} r_0 \log^{1-1/\alpha} r_0 }  \\
 & \leq \exp \Big( -(dt-\varepsilon) r_0 \log^{1-1/\alpha} r_0 -(1-\varepsilon)^{\alpha+1} \gamma (b-\varepsilon) t^{1-\alpha} r_0 \log^{1-1/\alpha} r_0 \Big) . \numberthis\label{2bounds} \end{align*}
Carrying out our limit and letting $\varepsilon \downarrow 0$ implies the upper bound in \eqref{first-hopLDP}. This finishes the proof of Theorem \ref{theorem-expectedhops}. \ProofEnde

\subsubsection{Proof of Corollary \ref{cor-A}}\label{sec-ProofCor}

The identity \eqref{koptlimit} follows immediately from Theorem~\ref{theorem-expectedhops}. As for \eqref{Aexprate}, let $k^*(r_0)$ be the smallest maximizer of $k \mapsto a_k(x_0)$, and let $r_0 \mapsto \kmax(r_0)$ satisfy the assumption of the corollary, i.e., $\kmax(r_0) \geq k^*(r_0)$. The lower bound easily follows from \eqref{first-hopLDP} by estimating $1/A(x_0)$ from below by the single summand $a_{k^*(r_0)}(x_0)$ and using \eqref{koptlimit}. As for an upper bound, we first write
\begin{align*}
    & \limsup_{r,r_0} \frac{1}{r_0 \log^{1-1/\alpha} r_0} \log \frac{1}{A(x_0)} \leq \limsup_{r,r_0} \frac{1}{r_0 \log^{1-1/\alpha} r_0} \log \Big( \sum_{k=1}^{\lfloor \frac 12 r_0 \rfloor} a_{k}(x_0) + \sum_{k=\lfloor \frac 12 r_0 \rfloor +1}^{\infty} a_{k}(x_0) \Big) \\ &= \max \Big\{ \limsup_{r,r_0} \frac{1}{\frac 12 r_0 \log^{1-1/\alpha} r_0} \log \Big( \sum_{k=1}^{\lfloor \frac 12 r_0 \rfloor} a_{k}(x_0) \Big) , \, \limsup_{r,r_0} \frac{1}{r_0 \log^{1-1/\alpha} r_0} \log \Big( \sum_{k=\lfloor \frac 12 r_0 \rfloor+1}^{\infty} a_{k}(x_0) \Big) \Big\}.
\end{align*}
Then the proof of \eqref{third-hopLDP} implies that there exists a constant $D>0$ such that we have
\[  \sum_{k=\lfloor \frac 12 r_0 \rfloor+1}^{\infty}  a_{k}(x_0)\leq \sum_{k=\lfloor \frac 12 r_0 \rfloor+1}^{\infty} (D r_0^d)^{-k} = \frac{(D r_0^d)^{-\lfloor \frac 12 r_0 \rfloor+1}}{1-\frac{1}{Dr_0^d}} \leq \exp(-(\smfrac 12-o(1))r_0 \log r_0), \] 
therefore 
\[ \limsup_{r,r_0} \frac{1}{r_0 \log^{1-1/\alpha} r_0} \log \Big( \sum_{k=\lfloor \frac 12 r_0 \rfloor+1}^{\infty} a_{k}(x_0) \Big) =-\infty. \]
Moreover, the assumption in Corollary~\ref{cor-A} that $\kmax(r_0) \geq k^*(r_0)$ for all sufficiently large $r_0>0$ together with \eqref{first-hopLDP} yields
\begin{align*} & \limsup_{r,r_0} \frac{1}{r_0 \log^{1-1/\alpha} r_0} \log \Big( \sum_{k=1}^{\lfloor \frac 12 r_0 \rfloor} a_{k}(x_0) \Big) \\ & \leq \limsup_{r,r_0} \frac{1}{r_0 \log^{1-1/\alpha} r_0} \log (\lfloor r_0/2 \rfloor) +  \limsup_{r,r_0}   \frac{1}{r_0 \log^{1-1/\alpha} r_0} \log a_{k^*}(x_0) = -( d t^* + \gamma b {t^*}^{1-\alpha} ), \end{align*} 
where we recall that $t^*=(b \gamma (\alpha-1)/d)^{1/\alpha}$ is the unique minimizer of $t \mapsto dt+t^{1-\alpha}$ on $(0,\infty)$, cf.~\eqref{easyminimum}. We conclude the upper bound in \eqref{Aexprate}. \ProofEnde

\subsubsection{Proof of Proposition  \ref{cor-straighthops}}\label{sec-straighthopproof} 

Let $\varepsilon,\delta>0$ be fixed. First, let us note that by the definition of $T_{x_0}$ and the fact that the behaviour of the Lebesgue measure restricted to $\overline{B_r(o)}$ is subexponential in our limit,  \eqref{r0expdecay} is equivalent to
\begin{align*} 
& \limsup_{r,r_0} \frac{1}{r_0 \log^{1-1/\alpha} r_0} \log \int_{D_{\varepsilon,\delta}(k(r_0),x_0)} \Big(\prod_{l=1}^{k(r_0)-1} \frac{\d x_l}{\Leb(B_r(o))}\Big)\, \e^{-\gamma \sum_{l=1}^{k(r_0)} g(x_{l-1},x_l)} \\  
< & \limsup_{r,r_0} \frac{1}{r_0 \log^{1-1/\alpha} r_0} \log a_{k(r_0)}(x_0) = -\big( dt^\ast+b\gamma {t^*}^{1-\alpha} \big), \numberthis\label{smallrate} 
\end{align*}
with $k(r_0) \sim t^* r_0 \log^{-1/\alpha} r_0$ and $x_{k(r_0)}=o$, where in the last step we used \eqref{first-hopLDP}. For this, it suffices to show that there exists $\varepsilon_1>0$ such that for any choice of $x_0\mapsto(x_1,\ldots,x_{k(r_0)-1})=(x_1(x_0),\ldots,x_{k(r_0)-1}(x_0)) \in D_{\varepsilon,\delta}(k(r_0),x_0))$ writing $I=I(x_0,x_1,\ldots,x_{k(r_0)-1})$ as in \eqref{depsilondelta}, we have
\[ \liminf_{r,r_0} \frac{\sum_{l=1}^{k(r_0)} g(x_{l-1},x_l)}{k(r_0) \log r_0} = \liminf_{r,r_0} \frac{\sum_{l=1}^{k(r_0)} g(x_{l-1},x_l)}{t^\ast r_0 \log^{1-1/\alpha} r_0} \geq b {t^*}^{-\alpha}+\varepsilon_1.  \numberthis\label{strictuniformbound} \]
Indeed, then one can argue analogously to \eqref{2bounds} to conclude the first inequality in \eqref{smallrate}.

Now we prove \eqref{strictuniformbound}. 
We will first replace the functional $(x,y) \mapsto  g(x,y)$ with $(x,y) \mapsto \frac{b}{\ell(|x-y|)}$ everywhere and then argue for $g(x,y)$, estimating the numerator of $g(\cdot,\cdot)$ similarly to Lemma~\ref{lemma-gapproxbound}.

We have, first using Jensen's inequality\index{Jensen's inequality} for the convex function $|\cdot|^\alpha$, then by the definition of $D_{\varepsilon,\delta}(k(r_0),x_0)$ together with the fact that $\alpha>1$,
\begin{align*} & \frac{1}{\# I} \sum_{l \in I} |x_l-x_{l-1}|^{\alpha} \geq  \Big(\frac{1}{\# I} \sum_{l \in I} |x_l-x_{l-1}|\Big)^{\alpha} \geq \Big( \frac{1}{\# I} \sum_{l \in I}  \big| |x_l|-|x_{l-1}|\big|)+\varepsilon \log^{1/\alpha} r_0 \Big)^{\alpha} \\ & \geq  \Big( \frac{1}{\# I} \sum_{l \in I}  \big||x_l|-|x_{l-1}|\big|\Big)^{\alpha}+(\varepsilon \log^{1/\alpha} r_0)^{\alpha}. \numberthis\label{startshere} \end{align*}
Similarly, by Jensen's inequality\index{Jensen's inequality} and the triangle inequality,
\[ \frac{\sum_{l \in [k(r_0)] \setminus I} |x_l-x_{l-1}|^{\alpha} }{k(r_0)-\# I}  \geq \Big(\frac{1}{k(r_0)-\# I} \sum_{l \in [k(r_0)] \setminus I} |x_l-x_{l-1}| \Big)^{\alpha} \geq \sum_{l \in [k(r_0)] \setminus I} \Big(\frac{1}{k(r_0)-\# I} \big||x_l|-|x_{l-1}|\big| \Big)^{\alpha}.\] 
Hence, more applications of Jensen's inequality\index{Jensen's inequality} yield
\begin{align*}
    & \frac{1}{k(r_0)} \sum_{l \in [k(r_0)]} |x_l-x_{l-1}|^{\alpha} \\
    &= \frac{\# I}{k(r_0)} \frac{1}{\# I} \sum_{l \in I}  |x_l-x_{l-1}|^{\alpha}  + \frac{k(r_0)-\# I}{k(r_0)} \frac{1}{k(r_0)-\# I} \sum_{l \in [k(r_0)] \setminus I} |x_l-x_{l-1}|^{\alpha}  
    \\ & \geq \frac{\# I}{k(r_0)} \Big( \frac{1}{\# I} \sum_{l \in I}  \big||x_l|-|x_{l-1}|\big|\Big)^{\alpha}  + \frac{k(r_0)-\# I}{k(r_0)} \Big(\frac{\sum_{l \in [k(r_0)] \setminus I} \big||x_l|-|x_{l-1}|\big|}{k(r_0)-\# I}  \Big)^{\alpha} +\frac{\# I(\varepsilon \log^{1/\alpha} r_0)^{\alpha}}{k(r_0)} 
    \\ & \geq \Big( \frac{1}{k(r_0)} \sum_{l \in [k(r_0)]} \big||x_{l-1}|-|x_{l}| \big|\Big)^{\alpha} +\delta \varepsilon^{\alpha} \log r_0 
    \geq \Big( \frac{r_0}{k(r_0)} \Big)^{\alpha} + \delta \varepsilon^{\alpha} \log r_0\\
    &=({t^*}^{-\alpha}+ \delta \varepsilon^{\alpha}) \log r_0, \numberthis\label{manyJensen}
\end{align*}
where in the penultimate step we used that $\# I \geq \delta k(r_0)$.

Now, we turn to $\ell(|\cdot|)$ instead of $|\cdot|^{-\alpha}$. Hence, we have to distinguish between $|\cdot|\leq 1$ and $|\cdot|>1$. Let us define $I' =I'(x_0,(x_1,\ldots,x_{k(r_0)-1})) \subseteq [k(r_0)]$ as the set of $l \in [k(r_0)]$ such that  $|x_l-x_{l-1}| \leq 1$. Without loss of generality, $I'$ is not empty. Then, after passing to a subsequence, if needed, we have that $\# I' \sim\delta' k(r_0)$ for some $\delta' \in [0,1]$.  Thus,
\[ \frac{1}{\# I'} \sum_{l \in I'} \frac{|x_{l-1}-x_{l}|-\big| |x_{l-1}|-|x_l|\big|}{\log^{1/\alpha} r_0}  = \Ocal(1/\log^{1/\alpha} r_0)=o(1). \numberthis\label{smallissmall} \]
Let us assume for a moment that $I \cap I' = \emptyset$ and $\delta'<1$. Splitting into $I'$ and $[k(r_0)]\setminus I'$, we obtain
\begin{equation}
\begin{aligned}
 \frac{1}{k(r_0)} \sum_{l \in [k(r_0)]}  \frac1{\ell(|x_l-x_{l-1}|)} 
 &\geq \frac{1}{k(r_0)} \Big( \Ocal(\# I') + \sum_{l \in [k(r_0)] \setminus I'} |x_l-x_{l-1}|^{\alpha}\Big) \\ 
 &\geq \delta'-o(1) + \frac{1-\delta'-o(1)}{k(r_0)-\# I'} \sum_{l \in [k(r_0)] \setminus I'} |x_l-x_{l-1}|^{\alpha}. \label{deltaprimeestimate}
\end{aligned}
\end{equation}
We want to apply to the last term a lower bound analogous to \eqref{manyJensen}, i.e., for the sum over $[k(r_0)] \setminus I'$ instead of $[k(r_0)]$. For this, we need that the sum of the $||x_{l-1}|-|x_l||$ satisfies a lower bound against $\sim r_0$. Using that $I \cap I' = \emptyset$, we indeed see this as follows:
\[ 
\sum_{l \in [k(r_0)] \setminus I'} \big||x_{l-1}|-|x_l|\big| \geq - (\delta'+o(1)) k(r_0)+ \sum_{l \in [k(r_0)]}\big||x_{l-1}|-|x_l|\big|   \geq r_0(1-o(1)).
\] 
Now, making a  computation analogous to \eqref{manyJensen} for the right-hand side of \eqref{deltaprimeestimate}, we obtain in our limit
\begin{align*} 
 \frac{1}{k(r_0)} \sum_{l \in [k(r_0)]} \frac1{\ell(|x_l-x_{l-1}|)} 
&\geq \delta'-o(1) + (1-\delta'-o(1))\Big[\Big( \frac{r_0}{\#([k(r_0)]\setminus I')} \Big)^{\alpha} + \frac\delta{1-\delta'} \varepsilon^{\alpha} \log r_0\Big]\\
&\geq \Big((1-\delta')^{1-\alpha} {t^\ast}^{-\alpha} +\delta \varepsilon^{\alpha}-o(1)\Big) \log r_0 
\geq  ({t^\ast}^{-\alpha} +\delta \varepsilon^{\alpha}-o(1)) \log r_0. 
\numberthis\label{ellbound} \end{align*}
The case $I \cap I' \neq \emptyset$ can be handled analogously as long as $\delta'<1$. Indeed, in this case \eqref{smallissmall} implies that $\liminf_{r,r_0} \# (I \setminus I')/k(r_0)$ and $\liminf_{r,r_0} \frac{1}{k(r_0)} \sum_{l \in I \setminus I'} (|x_{l-1}-x_l|-||x_{l-1}|-|x_l||)$ are positive. Thus, in our limit, a lower estimate on $\frac{1}{k(r_0)} \sum_{l \in [k(r_0)]} (|x_{l-1}-x_l|-||x_{l-1}|-|x_l||)$ can still be obtained analogously to \eqref{deltaprimeestimate}. Further, we observe that the corresponding lower bound on $\frac{1}{k(r_0)} \sum_{l \in [k(r_0)]} \frac1{\ell(|x_l-x_{l-1}|)} $ that is analogous to the first expression in the second line of \eqref{ellbound} tends to infinity as $\delta' \uparrow 1$. 

Hence, we have in any case that \eqref{ellbound} holds with $\delta \varepsilon^{\alpha}$ replaced by some positive number. From this, \eqref{strictuniformbound} follows for $(x,y) \mapsto g(x,y)$ replaced by $(x,y) \mapsto \frac{b}{\ell(|x-y|)}$ for some 
$\varepsilon_1>0$.

In order to conclude \eqref{strictuniformbound}, we now proceed similarly to the proof of \eqref{second-hopLDP}, that is, we use uniform convergence of the interferences to $b$ within $B_r(o)$ away from the boundary. Let us recall the auxiliary function $s$ and the index $K(r_0,r)$ at \eqref{k0def}. We essentially show that either a non-negligible part of the deviations from the straight line induced by the definition of $D_{\varepsilon,\delta}(k(r_0),x_0)$ takes place after the $K(r_0,r)$-th hop or the first $K(r_0,r)$ hops have a very high interference penalization value, and in both cases \eqref{strictuniformbound} holds.

For each $x_0$ with $|x_0|=r_0$, let us choose $ (x_1(x_0),\ldots,x_{k(r_0)-1}(x_0)) \in D_{\varepsilon,\delta}(k(r_0),x_0)$. Let us use the notation $\tau(r_0)=\tau(x_0,x_1(x_0),\ldots,x_{k(r_0)-1}(x_0))$ for
$ \tau(r_0)=\frac{K(r_0,r)}{k(r_0)}$.
Let us further write  $I(r_0)=I(x_0,x_1(x_0),\ldots,x_{k(r_0)-1}(x_0))$ for a choice of a set $I$ according to \eqref{depsilondelta}. According to \eqref{ellbound}, without loss of generality we can assume that $I'=I'(x_0,x_1(x_0),\ldots,x_{k(r_0)-1}(x_0))=\emptyset$ for all $x_0$ considered.

In our limit,  $\int_{B_r(o)}  \ell(|z-y|) \d z =b-o(1)$ uniformly in  $y \in \overline{B_{s(r)}(o)}$. Thus, in case $\tau(r_0)=0$, \eqref{ellbound} implies that \eqref{strictuniformbound} holds with some $\varepsilon_1>0$. Hence, in order to conclude \eqref{strictuniformbound}, we can assume that $\tau(r_0) \neq 0$ eventually in our limit. Further, by our assumptions on the function $s$, for any $\varepsilon_2>0$, eventually $s(r)>(1-\varepsilon_2)r_0$. Now, since $x_l(x_0) \in \overline{B_{s(r)}(o)}$ for all $l>K(r_0,r)$, similarly to \eqref{ellbound}, the convexity of $1/\ell(|\cdot|)$ implies the following
\begin{equation}\label{lesshops} 
\begin{aligned}
\frac{1}{k(r_0)} \sum_{l \in [k(r_0)]} g(x_{l-1},x_l) &\geq \frac{1}{k(r_0)} \sum_{l =K(r_0,r)+1}^{k(r_0)} g(x_{l-1},x_l)
\geq \frac{1}{k(r_0)} \sum_{l =K(r_0,r)+1}^{k(r_0)} \frac{b-o(1)}{\ell(|x_{l-1}-x_l|)}\\
&\geq \kappa(\varepsilon_2) (1-\tau(r_0))^{1-\alpha} (b-o(1)) {t^{*}}^{-\alpha} \log r_0 
\end{aligned}
\end{equation}
for some function $\kappa \colon [0,1] \to \R$ with $\lim_{\varrho\downarrow 0}\kappa(\varrho)=1$. Now, let $\eps_3>0$. Taking first our limit and then $\varepsilon_2 \downarrow 0$, we see that if $\liminf_{r,r_0} \tau(r_0)$ is at least $\eps_3$, then the proof of our goal \eqref{strictuniformbound} is finished; however, it is \emph{a priori} not clear that $\eps_1$ in \eqref{strictuniformbound} can be chosen uniformly bounded away from zero in the limit $\eps_3 \downarrow 0$. Therefore, we will now assume that $\liminf_{r,r_0} \tau(r_0)=0$; this is a case that has to be handled separately, and the computations corresponding to this case will also allow for handling the limit $\eps_3 \downarrow 0$ in the previous case. After passing to a subsequence, we can assume that $\lim_{r,r_0} \tau(r_0)=0$. 

Let us first investigate the case that $\limsup_{r,r_0}\frac 1{r_0} \sum_{l=K(r_0,r)+1}^{k(r_0)} (|x_{l-1}|-|x_l|)<1$. Then we have 
\[
\liminf_{r,r_0} \frac{1}{r_0} \sum_{l=1}^{K(r_0,r)} \big||x_{l-1}|-|x_l|\big|\geq \liminf_{r,r_0} \frac{1}{r_0} \sum_{l=1}^{K(r_0,r)} (|x_{l-1}|-|x_l|)>\varepsilon_4 
\]
for some $\varepsilon_4>0$ depending on $\limsup_{r,r_0}\frac 1{r_0} \sum_{l=K(r_0,r)+1}^{k(r_0)} (|x_{l-1}|-|x_l|)$. Thus, using that $\int_{B_r(o)} \ell(|z-y|) \d z \geq b_0-o(1)$ uniformly for $y \in \overline{B_r(o)}$ in our limit (where $b_0$ was defined before \eqref{fullintegral}), a convexity argument similar to \eqref{manyJensen} yields
\begin{equation}
\begin{aligned}
\liminf_{r,r_0}  \frac{1}{k(r_0)\log r_0} \sum_{l \in [k(r_0)]} g(x_{l-1},x_l) &\geq\liminf_{r,r_0} \frac{1}{{t^{\ast}}^{\alpha} k(r_0)^{1-\alpha} r_0^{\alpha} }\sum_{l=1}^{K(r_0,r)} g(x_{l-1},x_l) \\
&\geq \liminf_{r,r_0} \varepsilon_4^{\alpha} {t^{\ast}}^{-\alpha} (b_0-o(1)) \Big( \frac{K(r_0,r)}{k(r_0)} \Big)^{1-\alpha}=\infty. 
\end{aligned}
\end{equation}

Next, we fix $\eps_5 \in (0,\eps)$ and we consider the case that $\liminf_{r,r_0}\frac 1{r_0} \sum_{l=K(r_0,r)+1}^{k(r_0)} (|x_{l-1}|-|x_l|) \geq 1$ (observe that the total sum over all $l\in[k(r_0)]$ is telescoping and hence equal to $r_0$) and 
\begin{equation}\label{partialsum}
\liminf_{r,r_0}  \frac{1}{\# I(r_0)} \sum_{l \in I(r_0)\colon l>K(r_0,r)} \frac{|x_{l-1}-x_l|-\big||x_{l-1}|-|x_l|\big|}{\log^{1/\alpha}r_0}>\eps_5.
\end{equation}
Then one can employ an estimate analogous to \eqref{manyJensen} in order to conclude \eqref{strictuniformbound}.

Finally, we consider the case that $\liminf_{r,r_0} \frac{1}{r_0} \sum_{l=K(r_0,r)+1}^{k(r_0)} (|x_{l-1}|-|x_l|) \geq 1$ but  \eqref{partialsum} fails. After passing to a subsequence, we can assume that $\exists \lim_{r,r_0} w(r_0)\geq\eps-\varepsilon_5$, where we put
\[ 
w(r_0) = \frac{1}{\# I(r_0)} \sum_{l \in I(r_0) \cap [K(r_0,r)]} \frac{|x_{l-1}-x_l|-\big||x_{l-1}|-|x_l|\big|}{\log^{1/\alpha}r_0}.
\]
Using also that $\# I(r_0)\geq \delta k(r_0)\sim \delta t^* r_0\log^{1/\alpha} r_0$, we have
\[ \eps-\varepsilon_5-o(1) \leq  \frac{1}{\# I(r_0)} \sum_{l \in I(r_0) \cap [K(r_0,r)]} \frac{|x_{l-1}-x_l|}{\log^{1/\alpha}r_0} \leq \Big( \frac{1}{\delta t^{\ast}} + o(1) \Big) \sum_{l \in I(r_0) \cap [K(r_0,r)]} \frac{|x_{l-1}-x_l|}{r_0}. \]
Thus,  a convexity argument similar to \eqref{manyJensen} implies
\begin{align*} 
&  \liminf_{r,r_0} \frac{1}{k(r_0) \log r_0} \sum_{l \in [k(r_0)]} g(x_{l-1},x_l) \geq \liminf_{r,r_0} \frac{1}{k(r_0) \log r_0} \sum_{l \in I(r_0) \cap [K(r_0,r)]}   g(x_{l-1},x_l) \\ 
& \geq  \liminf_{r,r_0} \frac{1}{k(r_0) \log r_0} (b_0-o(1)) \Big(\frac{r_0 \delta t^{*}(\varepsilon-\eps_5)}{\# (I(r_0) \cap [K(r_0,r)])}\Big)^{\alpha}  \# (I(r_0) \cap [K(r_0,r)]) \\
&\geq  \liminf_{r,r_0} \tau(r_0)^{1-\alpha} b_0(\delta (\eps -\varepsilon_5))^{\alpha}
=\infty. 
\end{align*}
Hence, in case $\liminf_{r,r_0} \tau(r_0)=0$, \eqref{strictuniformbound} holds with a suitable choice of $\varepsilon_1>0$. Further, the computations corresponding to this case show that this $\eps_1$ can be chosen in such a way that as $\liminf_{r,r_0}\tau(r_0)$ tends to zero, we have a lower bound on $\frac{1}{k(r_0) \log r_0} \sum_{l \in [k(r_0)]} g(x_{l-1},x_l)$ the $\liminf$ of which does not exceed $\varepsilon_1$. This allows for handling the limit $\eps_3\downarrow 0$ in the earlier case that $\liminf_{r,r_0}\tau(r_0) \geq \eps_3$.

Taking infimum over the values of $\eps_1$ in \eqref{strictuniformbound} corresponding to all the different cases, we see that this infimum is positive. We conclude that \eqref{strictuniformbound} holds with a suitable choice of $\varepsilon_1>0$. \ProofEnde

\subsection{Discussion about the results of Section \ref{sec-largespace}} \label{sec-koptdiscussion}
This section discusses the relevance and extensions of the results of Section~\ref{sec-kopt}. In Section~\ref{sec-largespacediscussion} we interpret our large-distance limit, in Section~\ref{sec-pathlossdiscussion} we explain how the choice of the path-loss function influences our results, and in Section~\ref{sec-kmax=inftydiscussion} we comment on allowing an arbitrary number of hops.  

\subsubsection{The large-distance limit} \label{sec-largespacediscussion}

In Section~\ref{sec-kopt}, we consider the typical trajectory in a large homogeneous multihop communication system with one base station in the area $W$, after the high-density limit has been taken. According to the basic rules in this system, virtually every hop in the area $W$ is homogeneously admitted (even those that do not bring the message any closer to the base station or even further away), but an exponential interference weight is given to the joint configuration of all the trajectories. It may appear somewhat irrelevant to consider a limit of large area, large distances and many hops, since with an increasing number of hops the technical difficulties and annoying side-effects become larger, but our work is meant to reveal the basic effects emerging in such a setting, in particular the effect of the interference penalization, and our result in terms of a large deviation principle gives also bounds on deviations from the extreme regime.

Since the interference term in particular gives small weights to large hops, it may be expected that the typical trajectory turns out to follow a straight line with all the hops being of the same length, but it may also come as a surprise that the typical hop length diverges like a power of the logarithm of the distance. The reason for this is the fact that {\em a priori} all the hops (within the area) are admitted and that, in the distribution $T_{x_0}$ of the typical trajectory, as $W$ grows in size, a very small weight term $1/\Leb(W)=1/\mu(W)$ for each hop appears. This favours a small number of hops. The best compromise between this effect and the interference effect turns out to be on a logarithmic scale.

One could think of a model in which the search for the next hop is done only in a neighbourhood of the current location, which would presumably lead to the removal of the small weight term $1/\Leb(W)$ per hop and finally to a number of hops that is linear in the distance from the origin, but this would make the decay of the path-loss function $\ell$ irrelevant and describe a fundamentally different organization of message routeing in the telecommunication system. Such an organization is found e.g.~in the continuum percolation setting of \cite{YCG11}, where the optimal number of hops turns out to be asymptotically linear in the distance from the user to the origin in a large-distance limit. Further, \cite[Theorem 2.1]{YCG11} claims that the probability of having trajectories of a significantly unusual length decays exponentially fast, which is analogous to our Proposition \ref{cor-straighthops}.

\subsubsection{The role of the choice of the path-loss function in the large-distance limit}\label{sec-pathlossdiscussion}

We derived our large-distance statements for the path-loss function $\ell(r)=\min \{ 1, r^{-\alpha} \}$ for $\alpha>d$, since this $\ell$ describes the propagation of signal strength realistically, see e.g. \cite{BB09, GT08, HJKP15}. However, following the proofs of the results of Section~\ref{sec-kopt} presented in Section~\ref{sec-Aproof}, we conclude that analogous results hold whenever the path-loss function $\ell$ has the following two properties: $\int_{\R^d} \ell(|x|) \d x < \infty$ and $1/\ell$ is convex. If $\ell$ satisfies these assumptions, then in our large-distance limit, in the optimal strategy (cf.~Section~\ref{sec-Aproofmain}), the user takes $\asymp k(r_0)$ hops, where $r_0 \mapsto k(r_0)$ satisfies
\[ \log(r_0) \sim \ell\Big( \frac{r_0}{k(r_0)} \Big)^{-1}. \numberthis\label{ellscale} \]
This shows that the optimal scale depends only on the tail behaviour of $\ell$. Thus, for example, the results of Section~\ref{sec-kopt} also hold for the path-loss function $\ell(r)= (K+r)^{-\alpha}$, $K>0$, $\alpha>d$. In general, \eqref{ellscale} shows that under the two above assumptions on $\ell$, the optimal scale diverges to $\infty$ and is sublinear. The faster $\ell$ decays, the slower $r_0/k(r_0)$ grows.  E.g., if $\ell(r)=\e^{-\alpha r}$ for some $\alpha>0$, then the correct scale is $k(r_0) \asymp r_0/\log \log r_0$.

\subsubsection{Allowing an unbounded number of hops}\label{sec-kmax=inftydiscussion} 

Note that the interference term is linear in the number of hops, hence this number is upper bounded by some geometric random variable and thus almost surely finite, even without the upper bound $\kmax$. For a similar reason, the measures $\nu_k$ in \eqref{nukminimizerbeta=0} are also well-defined and are the unique minimizers of the variational formula \eqref{beta0variation} for $\kmax=\infty$. Also, it is clear that the proof of Corollary~\ref{cor-A} also works if we choose $\kmax=\infty$ for each $r_0$. However, the proof techniques of \cite[Proposition 2.1]{KT17} do not generalize to the case $\kmax=\infty$ or $\kmax$ being a function of $\lambda$ and tending to infinity as $\lambda \to \infty$. Thus, as long as an analogue of this proposition has not been proven for $\kmax=\infty$, these results have no verified connection with a Gibbsian model. Proving such an analogue may be a mathematically interesting task, nevertheless, from a modelling point of view, we find the necessity of taking an arbitrary 
number of hops in a fixed compact communication area questionable.

\section{Strong penalization for the interference}\label{sec-gammatoinfty} 
This section is devoted to regime~\eqref{regime2}, i.e., the limit of strong penalization of interference. Our main result corresponding to this, Proposition~\ref{prop-gammatoinfty}, is stated in Section~\ref{sec-gammatoinftystatement} and proven in Section~\ref{sec-gammatoinftyproof}.
\subsection{Strong interference penalization makes message trajectories straight} \label{sec-gammatoinftystatement}

Proposition \ref{cor-straighthops} shows that in the large-distance limit, with $\mu$ being the Lebesgue measure in a large ball $W$, the typical message trajectory from the transmitter $x_0$ to $x_k=o$ under $T_{x_0}$ does not deviate much from the straight line with high probability. In this proposition, $|x_0|$, $k=k(|x_0|)$ and the radius of $W$ are assumed to tend to infinity in a certain coupled way. From an application point of view, it is also desirable to see a similar effect for a fixed compact communication area $W$, a fixed starting site $x_0$ and a fixed upper bound $\kmax \in \N$ on the hop number. One way to find such an effect is to consider the limit of a large interference penalization parameter $\gamma$. It is easily seen from \eqref{Tident} that this limiting behaviour should be entirely described by the minimizer of $W^{k-1} \ni (x_1,\ldots,x_{k-1}) \mapsto \sum_{l=1}^{k}g(x_{l-1},x_l)$. In this section, for $k \in [\kmax]$, we write $\nu_k^\gamma$ for the measure $\nu_k$ introduced in \
eqref{nukminimizerbeta=0} and $T_{x_0}^\gamma$ for the measure $T_{x_0}$ defined in \eqref{Tdef} corresponding to the parameter $\gamma$. Our next result gives criteria under which this minimizer follows a straight line and we have exponential estimates for deviations of trajectories from that.

Let us consider the case where $W$ is a closed ball $\overline{B_r(o)}$ with $r>0$, and the path-loss function $\ell$ is strictly monotone decreasing (and satisfies the original condition that it is continuous and positive on $[0,\infty)$). A typical choice \cite[Section 22.1.2]{BB09}  is $\ell(r)=(1+r)^{-\alpha}$. Further, let us assume that the intensity measure is rotationally invariant, i.e., $\mu \circ O^{-1} = \mu$ for any orthogonal $d \times d$ matrix $O$. Under these conditions, we conclude that any minimizer of $W^{k-1} \ni (x_1,\ldots,x_{k-1}) \mapsto \sum_{l=1}^{k}g(x_{l-1},x_l)$ is of the form $x_l=c_l x_0$ for $l=1,\ldots,k-1$ with positive constants $1 > c_1 > \ldots > c_{k-1}>0$. Moreover, the total probability mass carried by trajectories deviating from the straight line through the transmitter and $o$ in Euclidean distance at least by some fixed positive quantity decays exponentially fast as $\gamma \to \infty$.

More precisely, writing $[[x,y]]=\lbrace \alpha x +(1-\alpha) y \vert \alpha \in \R \rbrace$ for the line through $x,y \in \R^d$, we state the following.

\begin{prop} \label{prop-gammatoinfty} Let $r>0, W=\overline{B_r(o)}, \kmax \geq 2, \ell$ and $\mu$ be fixed. Let us assume that $\ell$ is strictly monotone decreasing and $\mu$ is rotationally invariant. 
\begin{enumerate}[(A)]
\item \label{first-gammatoinfty} For $x_0 \in W$, let us write
\[ \m_{\kmax}(x_0) = \min_{k \in [\kmax]} \min_{x_1,\ldots,x_{k-1} \in W} \sum_{l=1}^{k}g(x_{l-1},x_l),\qquad x_k=o.\] 
Then, for any minimizer $k \in [\kmax]$ and $x_1,\ldots,x_{k-1}$, there exist $1>c_1>\ldots>c_{k-1}>0$ such that $x_l=c_l x_0$ for all $l \in [k-1]$.

\item \label{second-gammatoinfty} For $k \in [\kmax]$ and $\varepsilon>0$, let us define
\[ D^\varepsilon_k(x_0) = \lbrace (x_1,\ldots,x_{k-1}) \in W^{k-1} \mid ~\exists l \in \lbrace 1,\ldots,k-1 \rbrace \colon \dist(x_l,[[x_0,o]])>\varepsilon \rbrace. \numberthis\label{penaltyminimum} \]
Then, we have
\[ \sup_{x_0 \in W} \sup_{k \in [\kmax]} \limsup_{\gamma \to \infty} \frac{1}{\gamma} \log \sup_{(x_1,\ldots,x_{k-1}) \in D_k^\varepsilon(x_0)} T^\gamma_{x_0}(k,x_1,\ldots,x_{k-1}) < 0.  \numberthis\label{gammaexpdecay} \]
\end{enumerate}
\end{prop}

The proof of the first part of this proposition is based on simple geometric arguments, while the proof of the second part additionally uses the Laplace method. Note that in the first part, a minimizer always exists because $W$ is compact and $g$ is continuous. The proof is carried out in Section \ref{sec-gammatoinftyproof}. 

We expect that Proposition~\ref{prop-gammatoinfty} \eqref{first-gammatoinfty} is not true in general if $\ell$ is not strictly monotone decreasing. Indeed, in this case, modifying the position of a relay in a path that is optimal with respect to interference penalization may not change the penalization at all. This indicates that if $\mathfrak m_{\kmax}(x_0)$ is attained for some $x_0$ by a path along $[x_0,o]$, it may also be attained by a non-straight path.

\subsection{Proof of Proposition \ref{prop-gammatoinfty}} \label{sec-gammatoinftyproof}

Throughout the proof, given any number of hops $k \in [\kmax]$, we will always assume that $x_k=o$. 

We start with proving part \eqref{first-gammatoinfty}. Let us fix $x_0 \in \overline{B_r(o)}$. The fact that $(x,y) \mapsto g(x,y)$ is bounded away from 0 implies that for $x_0=o$, $\mathfrak m_{\kmax}(x_0)$ is uniquely attained at the 1-hop trajectory from $x_0$ to $x_1=o$. Thus, we can assume that $x_0 \neq o$.

Let now $k \in [\kmax]$ and $(x_1,\ldots,x_{k-1}) \in \overline{B_r(o)}^{k-1}$. Let us assume that $\sum_{l=1}^{k} g(x_{l-1},x_l) = \mathfrak m_{\kmax}(x_0)$. We show that there are $1>c_1>\ldots>c_{k-1}>0$ such that $x_j=c_j x_0$ for all $j \in [k-1]$, proceeding in the following steps.
\begin{enumerate}[(i)]
    \item Let $\mathcal H$ denote the closed half-space of $\R^d$ that contains $x_0$ and whose boundary is orthogonal to the vector from $x_0$ to $o$ and contains $o$. Then $(x_1,\ldots,x_{k-1}) \in \mathcal H^{k-1}$.
    \item $(x_1,\ldots,x_{k-1}) \in (\mathcal H \cap [x_0,o])^{k-1}$, where we write $[x,y]=\lbrace \alpha x + (1-\alpha) y \colon \alpha \in [0,1] \rbrace$ for the closed segment between $x,y \in \R^d$.
    \item $|x_0|>|x_1| >\ldots> |x_{k-1}|>0$.
\end{enumerate}
We prove these claims respectively as follows.

\begin{enumerate}[(i)]
    \item Assume that the assertion does not hold, then let us define another trajectory $(x_1',\ldots,x'_{k-1}) \in \mathcal H^{k-1}$ via $x_l'=x_l$ if $x_l \in \mathcal H$ and $x_l'$ being the image of $x_l$ under reflection across the boundary hyperplane of $\mathcal H$ otherwise, for all $l \in [k-1]$. The rotation invariance of $\mu$ and $W$, combined with $|x_l|=|x_l'|$, implies that
\[ \int_{W} \mu(\d y) \ell(|x_l-y|) = \int_{W} \mu(\d y) \ell(|x'_{l}-y|), \qquad \forall l\in[\kmax].\numberthis\label{numerator} \]
But, since $|x_{l-1}-x_l| \geq |x'_{l-1}-x'_l|$ and $\ell$ is strictly decreasing, 
\[ \ell(|x_{l-1}-x_l|) \leq \ell(|x'_{l-1}-x'_l|), \numberthis\label{denominator} \]
where equality holds if and only if $x_{l-1},x_l$ are both in $\mathcal H$ or both in $\R^d \setminus \mathcal H$. We conclude that $\sum_{l=1}^{k} g(x_{l-1},x_l) > \sum_{l=1}^k g(x'_{l-1},x'_l)$, which contradicts $(x_1,\ldots,x_{k-1})$ being the minimizer in \eqref{penaltyminimum}. 

   \item The case $d=1$ is trivial. Let us consider the case $d \geq 2$. Assume  $(x_1,\ldots,x_{k-1}) \in \mathcal H^{k-1}$. Let us define another trajectory $(x_1',\ldots,x'_{k-1}) \in (\mathcal H \cap [x_0,o])^{k-1}$ such that for all $l \in [k-1]$, $x_l'$ satisfies $|x_l'|=|x_l|$ and $x_l' \in [x_0,o]$. That is, $x'_l=x_0 |x_l|/|x_0|$. Then, the radial symmetry of $\mu$ implies that \eqref{numerator} holds. Furthermore, the fact that $\ell$ is strictly decreasing but  $|x_{l-1}-x_l| \geq |x'_{l-1}-x'_l|$ implies that also \eqref{denominator} is true in this case, where equality holds if and only if $x_l=x_l'$ for all $l \in [k-1]$, i.e., if $x_l \in [x_0,o]$ for all $l \in [k-1]$. 

      \item Let $(x_1,\ldots,x_{k-1}) \in [x_0,o]^{k-1}$. In the following argument, we cancel in this trajectory all hops that increase the distance from $o$. This results in a smaller sum of the interference terms. Indeed, let us define $i_0=0$ and $i_j=\inf \lbrace l \in [k] \colon |x_l|<|x_{i_{j-1}}|\rbrace$, $j=1,\ldots,k$. Let $m$ be the largest index $j$ such that $i_j <\infty$, then it is clear that $1 \leq m \leq k$ since $x_0 \neq o$. Now, let us define an $m$-hop trajectory with relay sequence $(y_1,\ldots,y_{m-1})=(x_{i_1},\ldots,x_{i_{m-1}})$, writing $y_0=x_0$ and $y_m=o$. Let us further define $\varepsilon'=\min_{x,y\in\overline{B_r(o)}} g(x,y)>0$. Then, since for any $j \in [m-1]$ we have that $|x_{i_j-1}-x_{i_j}| \geq |x_{i_{j-1}}-x_{i_j}|$, we conclude that 
   \[ \sum_{j=1}^{m} g(y_{j-1},y_j)=\sum_{j=1}^m g(x_{i_{j-1}},x_{i_j}) \leq \sum_{j=1}^{m} g(x_{i_j-1},x_{i_j}) \leq  \sum_{l=1}^{k} g(x_{l-1},x_{l})-(k-m)\varepsilon'. \]
   Thus, $(x_1,\ldots,x_{k-1})$ can only minimize \eqref{penaltyminimum} if $k=m$, that is, if $|x_0|>|x_1|> \ldots> |x_{k-1}|>0$.
\end{enumerate}

This finishes the proof of part \eqref{first-gammatoinfty} of Proposition \ref{prop-gammatoinfty}.

As for part \eqref{second-gammatoinfty}, we note that the case $d=1$ is trivial since $D_k^\varepsilon(x_0) = \emptyset$ for all $x_0 \in \overline{B_r(o)}$. Throughout the rest of the proof, let $d \geq 2$. First, we fix $x_0 \in \overline{B_r(o)}$ and $k \in [\kmax]$, and we verify that
\[ \limsup_{\gamma \to \infty} \frac{1}{\gamma} \log \sup_{(x_1,\ldots,x_{k-1}) \in D_k^\varepsilon(x_0)} T^\gamma_{x_0}(k,x_1,\ldots,x_{k-1}) < -\kappa \numberthis\label{fixedx0k} \]
for some $\kappa>0$ that neither depends on $x_0$ nor on $k$. This will imply \eqref{gammaexpdecay}.

Again, it is easy to see that if $x_0=o$, then \eqref{fixedx0k} holds for some $\kappa>0$, let us therefore assume that $x_0 \neq o$. We first verify that there exists $\delta=\delta(\varepsilon)>0$, independent of $x_0$ and $k$, such that
\[ \m^\varepsilon_{\kmax}(x_0) =  \inf_{(x_1,\ldots,x_{k-1}) \in D_k^\varepsilon(x_0)} \sum_{l=1}^k g(x_{l-1},x_l) \geq \m_{\kmax}+\delta(\varepsilon). \numberthis\label{Depsilonmax} \]

In the construction of $(x_1,\ldots,x_{k-1}) \mapsto (x'_1,\ldots,x'_{k-1})$ in the proof of (i) above, the fact that $\mathrm{dist}(x_l,[[x_0,o]])=\mathrm{dist}(x_l',[[x_0,o]])$ for all $l \in [k-1]$ and $k \in [\kmax]$ implies that if  $(x_1,\ldots,x_{k-1}) \in D_k^\varepsilon(x_0)$, then $(x'_1,\ldots,x'_{k-1}) \in D_k^\varepsilon(x_0) \cap \mathcal H^{k-1}$. It follows that the infimum in \eqref{Depsilonmax} can be realized along sequences of trajectories that have all their relays $x_1,\ldots,x_{k-1}$ in $\mathcal H$.

Let now $(x_1,\ldots,x_{k-1}) \in D_k^\varepsilon(x_0) \cap \mathcal H^{k-1}$, and consider the construction of $(x_1,\ldots,x_{k-1}) \mapsto (x'_1,\ldots,x'_{k-1})$ in the proof of (ii) above. We observe the following. Since $x_0 \in [x_0,o]$ and $(x_1,\ldots,x_{k-1}) \in D_k^\varepsilon(x_0)$, there exists $l_1 \in [k]$ such that \[ \mathrm{dist}(x_{l_1},[[x_0,o]])>\mathrm{dist}(x_{l_1-1},[[x_0,o]])+\frac{\varepsilon}{k} \geq \mathrm{dist}(x_{l_1-1},[[x_0,o]])+\frac{\varepsilon}{\kmax}, \]
where each $[[x_0,o]]$ can also be replaced by $[x_0,o]$. One easily sees that this bound holds uniformly in $x_0 \in W$ and $k \in [\kmax]$.

Now, the Pythagoras theorem together with the fact that $\ell$ is strictly monotone decreasing yields that in this case there exists $\delta'(\varepsilon)>0$ such that $\ell(|x_{l_1-1}-x_{l_1}|) < \ell(|x'_{l_1-1}-x'_{l_1}|)-\delta'(\varepsilon)$. Note that $\delta'(\varepsilon)$ depends only on $\ell$, $r$ and $\varepsilon$ but not on $k$ or $l_1$. On the other hand, by the rotational symmetry of $\mu$, the identity \eqref{numerator} holds for all $l \in [k]$ for this choice of the relays $x_l$ and $x_l'$. Therefore, we conclude that there exists a constant $\delta=\delta(\varepsilon)>0$ such that for all $n \in \N$ we have
\[ \sum_{l=1}^{k} g(x_{l-1},x_{l}) > \sum_{l=1}^{k} g(x_{l-1}',x_{l}')+\delta(\varepsilon) \geq \mathfrak m_{\kmax} +\delta(\varepsilon). \]
This implies \eqref{Depsilonmax}, and the construction shows that $\delta(\varepsilon)>0$ can be chosen independently of $x_0$ and $k$.

We now finish the proof of part \eqref{second-gammatoinfty}. Let us use the notation $A^{\gamma}(x_0)=A(x_0)$ for the normalization term in \eqref{Adefnew} corresponding to $\gamma$ and recall the notation $T^\gamma_{x_0}=T_{x_0}$ from Proposition \ref{prop-gammatoinfty}. It is clear from the Laplace method \cite[Section 4.3]{DZ98} that we have
\[ A^\gamma(x_0) = \e^{\gamma \m_{\kmax} (x_0)+o(\gamma)} \quad \text{as }\gamma \to \infty.\]
For any $(x_1,\ldots,x_{k-1}) \in D_k^\varepsilon$, using \eqref{nukminimizerbeta=0} and \eqref{Depsilonmax}, we can estimate
\[ T_{x_0}^\gamma(k, x_1,\ldots,x_{k-1}) = \frac{\nu_k^\gamma(\d x_0,\ldots,\d x_{k-1})}{\mu(\d x_0) \mu(\d x_1) \ldots \mu(\d x_{k-1})} \leq \e^{ \gamma \m_{\kmax}(x_0)-\gamma \m_{\max}^\varepsilon(x_0)+o(\gamma)} \leq \e^{ o(\gamma)-\gamma \delta(\varepsilon)}.  \]
We conclude \eqref{fixedx0k} (with $\kappa>0$ being independent of $x_0 \in W$ and $k \in [\kmax]$). Thus, part \eqref{second-gammatoinfty} of Proposition \ref{prop-gammatoinfty} follows.

\section{High local density of users} \label{sec-largea}
This section describes the behaviour of the system in regime~\eqref{regime3}, i.e., in the limit of a high local density of users in a subset of the communication area. We explain both global and local aspects of this limit, respectively in Section~\ref{sec-global} and Section~\ref{sec-local}. 

We consider the following question about the behaviour of our model given by \eqref{nukminimizerbeta=0}, assuming always that $\kmax \geq 2$. 
\begin{quote} Does the density of trajectories increase unboundedly in a densely populated subarea, or do the messages avoid such an area for the sake of having lower interference? 
\end{quote} 


In order to give substance to this question, we replace our user density measure $\mu$ by 
\begin{equation}
\mu^a= \mu+a \Leb|_\Delta\in\Mcal(W),\qquad a\in(0,\infty),
\end{equation}
where $\Leb|_\Delta$ is the Lebesgue measure concentrated on a compact set $\Delta\subseteq W$, seen as a measure on $W$, where we assume that $\Leb(\Delta)>0$. We think of $\Delta$ as of a set of very high concentration of users and will consider the behaviour of the optimal path trajectory in the limit $a\to\infty$. We will from now on label all objects that depend on $\mu^a$ instead of $\mu$ with the index $a$. 
We will study the measure 
\begin{equation}
M^a=\sum_{k=1}^{\kmax}\sum_{l=1}^{k-1} \pi_l \nu_{k}^a,
\end{equation}
where $\nu_k^a$ is defined according to \eqref{nukminimizerbeta=0}. It can be interpreted as the measure of all the incoming hops at a given location (see also Section~\ref{sec-typicaltrajectory}). Note that the total mass $M^a(W)$ is zero if all messages go directly to the base station without any relaying hop; hence it is a measure for the total amount of relaying hops. Explicitly, we have
\begin{equation} 
\label{MFormula}
M^a(\d x)  = \mu^a(\d x) \int_{W} \mu^a(\d x_0) \frac{\sum_{k=1}^{\kmax} \sum_{l=1}^{k-1} \int_{W^{k-2}} \prod_{l' \in [k-1] \setminus \{ l \}} \mu^a(\d x_{l'}) \,\e^{ -\gamma \sum_{l'=1}^{k} g^a(x_{l'-1},x_{l'}) \big|_{x_l=x} } }
{\sum_{k=1}^{\kmax} \int_{W^{k-1}} \prod_{l=1}^{k-1} \mu^a(\d x_{l}) \e^{-\gamma \sum_{l=1}^{k-1} g^a(x_{l-1},x_l) }}. 
\end{equation}
Now we are interested in the behaviour of the measure $M^a$ as $a \to \infty$. 
Since $(x,y)\mapsto \ell(|x-y|)$ is bounded away from 0 on $W\times W$, we first note that the large-$a$ behaviour of the interference term is given by
\[ 
\lim_{a \to \infty} \frac{1}{a} g^a(x,y)=\frac{\int_\Delta \d z\, \ell(|y-z|) }{\ell(|x-y|)}=:g_\Delta(x,y), \quad x,y \in W. \numberthis\label{gdeltadef}
\] 
The limiting function $g_\Delta$ measures the interference only in relation with the interference coming from $\Delta$. This ratio will turn out to be relevant and the effective interference term in the limit $a\to\infty$.

\subsection{Global effects}\label{sec-global} Our first result is that, when the path-loss function $(x,y) \mapsto \ell(|x-y|)$ does not vary much on $W \times W$, the presence of the highly dense area $\Delta$ has a strongly repellent effect everywhere in the system and suppresses all the relaying hops; indeed, the total mass of the measure $M^a$ tends to zero exponentially fast as $a\to\infty$, under our general assumptions on the path-loss function $\ell$. 

\begin{prop}[Criterion for exponential decay of the amount of relays] \label{prop-Ma(W)tozero}
We have \[ \sup_{x \in W} \limsup_{a \to \infty}  \frac{1}{a} \log \frac{\d M^a}{\d \Leb}(x) < 0 \numberthis\label{afullexpdecay} \]
if and only if
 \[ \min_{x_0\in W} \Big[\min_{x_1 \in W} \big(g_{\Delta}(x_0,x_1)+g_{\Delta}(x_1,o)\big)-g_{\Delta}(x_0,o)\Big]>0. \numberthis\label{twohop} \] 
\end{prop}

\begin{remark}\label{rem-Ma(W)}
\begin{enumerate}[(i)]
\item The inequality \eqref{afullexpdecay} implies an exponential decay of the total mass of $M^a$, i.e., 
\[ \limsup_{a \to \infty} \frac{1}{a} \log M^a(W) < 0. \]

\item Since $\mu^a$ is clearly subexponential in  $a \to \infty$,  \eqref{afullexpdecay} is equivalent to a uniform exponential decay of the Radon--Nikodym derivative of $M^a$ with respect to $\mu^a$ instead of $\Leb|_W$.

\item The condition in \eqref{twohop} says that the effective interference penalty for a two-hop trajectory is uniformly worse than the one of a direct hop to the origin. This condition involves only one- and two-hop trajectories and is valid even when $\kmax$ is much larger than 2. 

\item Multiplying by two of the three denominators in \eqref{twohop} and using that the map $W\times W\ni(x,y)\mapsto \ell(|x-y|)$ is bounded and bounded away from zero, we easily see that \eqref{twohop} holds if and only if 
\[  \min_{x_0,x_1 \in W} \Big[  \ell(|x_1|) \int_{\Delta} \ell(|z-x_1|) \,\d z   + \ell(|x_0-x_1|)  \int_{\Delta} \ell(|z|)\,\d z - \frac{\ell(|x_1|) \ell(|x_0-x_1|)}{\ell(|x_0|)} \int_{\Delta} \ell(|z|)\d z \Big]  > 0. \numberthis\label{allnumerator} \]

\item A sufficient condition for \eqref{twohop} to hold is as follows. 
Let $p \in (0,1]$ be such that $p \ellmax= \ellmin$, where 
$\ellmax=\max_{x,y\in W}\ell(|x-y|)$ and $\ellmin=\min_{x,y\in W}\ell(|x-y|)$ are the maximal and the minimal path-loss values in the system, respectively. Then, a lower bound for the left-hand side of \eqref{allnumerator} is $\ellmax^2 \Leb(\Delta)(2 p^2-\frac{1}{p})$. 
This is positive as long as $p$ is larger than $2^{-1/3} \approx 0.794$.

Similarly, an upper bound on the left-hand side of \eqref{allnumerator} in terms of $p$ is $\ellmax^2\Leb(\Delta)(2-p^3)$, but this is larger than zero for all $p \in (0,1]$, hence such a general estimate cannot be used for disproving \eqref{allnumerator} in any case.
\item In our numerical results in Examples~\ref{example-1dim} and \ref{example-2dim} with $W=\Delta$, the condition~\eqref{twohop} does not hold.
\end{enumerate}
\end{remark}
\noindent\emph{Proof of Proposition~\ref{prop-Ma(W)tozero}.}
Consider the quantity on the left-hand side of \eqref{afullexpdecay}. Taking the limit $a \to \infty$, we obtain for fixed $x,x_0 \in W$ for the numerator of \eqref{MFormula}
\begin{equation}\label{numeratorlargemu}
\begin{aligned}
\lim_{a \to \infty} \frac{1}{a} &\log \Big[ \sum_{k=1}^{\kmax} \sum_{l=1}^{k-1} \int_{W^{k-2}} \prod_{l' \in [k-1] \setminus \{ l \}} \mu^a(\d x_{l'}) \exp \Big( -\gamma \sum_{l'=1}^{k} g^a(x_{l'-1},x_{l'})\Big|_{x_l=x} \Big) \Big] \\ 
&=  -\gamma \min_{k \in [\kmax]\setminus\{1\}} \min_{l \in [k-1]} \min_{x_1,\ldots,x_{l-1},x_{l+1},\ldots,x_{k-1} \in W}  \sum_{l'=1}^{k} g_{\Delta}(x_{l'-1},x_{l'}) \Big|_{x_l=x}.
\end{aligned}
\end{equation} 
On the other hand, for the denominator of \eqref{MFormula} for $x_0$ fixed, we have
\begin{equation}\label{denominatorlargemu} 
\begin{aligned}
\lim_{a \to \infty} \frac{1}{a} &\log \Big[ \sum_{k=1}^{\kmax} \int_{W^{k-1}} \prod_{l=1}^{k-1} \mu^a(\d x_l) \,\exp \Big( -\gamma \sum_{l=1}^{k-1} g^a(x_{l-1},x_l) \Big)\Big] \\ 
&=  -\gamma \min_{k \in [\kmax]}  \min_{x_1,\ldots,x_{k-1} \in W} \sum_{l=1}^{k} g_{\Delta}(x_{l-1},x_l) .
\end{aligned}
\end{equation} 

These two assertions follow from the Laplace method \cite[Section 4.3]{DZ98} in a standard way, since the $a$-dependence of the integrating measure $\mu^a$ is clearly subexponential. Hence, we obtain that
\begin{equation} \label{kxminimization}
\begin{aligned}
\lim_{a\to\infty}\frac 1a\log M^a(\d x)=-\gamma\min_{x_0\in W}\Big[&\min_{k \in [\kmax]\setminus\{1\}} \min_{l \in [k-1]} \min_{x_1,\ldots,x_{l-1},x_{l+1},\ldots,x_{k-1} \in W}  \sum_{l'=1}^{k} g_{\Delta}(x_{l'-1},x_{l'}) \Big|_{x_l=x}\\
&-\min_{k \in [\kmax]}  \min_{x_1,\ldots,x_{k-1} \in W} \sum_{l=1}^{k} g_{\Delta}(x_{l-1},x_l)\Big].
\end{aligned}
\end{equation}

Note that after taking supremum over $x\in W$ on the right-hand side of \eqref{kxminimization}, we obtain a negative number if and only if 
\begin{equation}\label{contradiction}
\min_{x_0 \in W} \Big[ \min_{k \in [\kmax] \setminus \{ 1 \} } \min_{x_1,\dots,x_{k-1} \in W}\sum_{l=1}^k g_{\Delta}(x_{l-1},x_l)-g_{\Delta}(x_0,o) \Big] > 0. 
\end{equation}
Now, assume that the condition \eqref{twohop} does not hold. Then we may pick $x_0',x_1' \in W$ with $(g_{\Delta}(x_0',x_1')+g_{\Delta}(x_1',o)-g_{\Delta}(x_0',o)) \leq 0$. But this implies that \eqref{contradiction} is false, as is shown by taking $k=2$, $x_0=x_0'$ and $x=x_1'$. We conclude that \eqref{afullexpdecay} does not hold.


Conversely, let us assume that \eqref{afullexpdecay} is not satisfied and let us conclude that \eqref{twohop} also does not hold. Using \eqref{afullexpdecay} and \eqref{kxminimization}, we can choose $x_0 \in W$, $k \in [\kmax] \setminus \{ 1 \}$ and $x_1,\ldots,x_{k-1} \in W$ such that 
\[
\sum_{l=1}^{k} g_{\Delta}(x_{l-1},x_l)\leq g_{\Delta}(x_0,o), \qquad x_k=o. \numberthis\label{directnotgood} 
\] 
Let $k$ be minimal for $x_0$ with this property. We show that there exists $x_0' ,x_1' \in W$ such that $g_{\Delta}(x_0',x_1')+g_{\Delta}(x_1',o)\leq g_{\Delta}(x_0',o)$, therefore \eqref{twohop} does not hold. Indeed, if this is not the case for $x_0'=x_{k-2}$ and $x_1'=x_{k-1}$, then we have
\[ 
\sum_{l=1}^{k-2} g_{\Delta}(x_{l-1},x_l) + g_{\Delta}(x_{k-2},o) \leq  \sum_{l=1}^{k} g_{\Delta}(x_{l-1},x_l) \leq g_{\Delta}(x_0,o) < \sum_{l=1}^{k-2} g_{\Delta}(x_{l-1},x_l)+ g_{\Delta}(x_{k-2},o), \label{direct-2} 
\] 
where in the last step we used the minimality of $k$ for $x_0$. This is in contradiction with \eqref{twohop} and thus the proof is concluded.
\ProofEnde


\subsection{Local effects}\label{sec-local} The condition~\eqref{twohop} can be applied to any $\Delta \subseteq W$ with $\Leb(\Delta)>0$, in particular also to $\Delta=W$. In this sense, Proposition~\ref{prop-Ma(W)tozero} is non-spatial. We now discuss among what conditions the spatial effect that the quality of service (interference penalization with interference coming only from $\Delta$) is significantly worse for messages relaying through a neighbourhood of $\Delta$ than through an area sufficiently far away from $\Delta$ occurs in our model. For simplicity, we consider only the case  $\kmax=2$, a very small set $\Delta$ and a special choice of the path-loss function. We will give arguments that suggest that, for any large $a$, it is strictly suboptimal to relay through a neighbourhood of $\Delta$ as opposed to circumventing $\Delta$ sufficiently far.

Analogously to \eqref{numeratorlargemu}--\eqref{denominatorlargemu}, the large-$a$ limit for the mass of all relaying hops from $x_0$ into a set $A\subset W$ (assumed being equal to the closure of its interior) and further to $o$ is given by
\[ -\lim_{a \to \infty} \frac{1}{a} \log T^a_{x_0}(2,A)= \gamma \Big[\Xi_{x_0}(A) -\min \big\{g_{\Delta}(x_0,o) ,\Xi_{x_0}(W)\big\} \Big], \numberthis\label{minimumspeltout} \]
where 
$$
\Xi_{x_0}(A)=\min_{x_1 \in  A} [g_{\Delta}(x_0,x_1)+g_{\Delta}(x_1,o)].
$$
We want to discuss under what conditions $\Xi_{x_0}(A)$ is smaller for sets $A$ that are bounded away from $\Delta$ than for $A$ being a neighbourhood of $\Delta$. For simplicity, let us do that for $W=\R^d$ and very small sets $\Delta=B_r(y_0)$ with $r\ll 1$ only, i.e., we approximate 
\begin{equation}\label{appr}
g_\Delta(x,y)\approx |\Delta| \frac{\ell(|y-y_0|)}{\ell (|y-x|)},\qquad x,y\in\R^d.
\end{equation}
Hence, we will put $\Delta=\{y_0\}$ and discuss the function
$$
f_{x_0,y_0}(\eps)=\min_{x_1 \in W\colon |x_1-y_0|= \eps} \Big[\frac{\ell(|x_1-y_0|)}
{\ell(|x_0-x_1|)} + \frac{\ell(|y_0|)}{\ell(|x_1|)}\Big],\qquad\eps\geq 0.
$$
This is an approximation of $\Xi_{x_0}(\partial B_\eps(y_0))$. We will see that, under quite general conditions, $f_{x_0,y_0}(\eps)<f_{x_0,y_0}(0)$ for all  $\eps\in[0,\eps_0]$ for some $\eps_0>0$. This means that, for all sufficiently large $a$, the probability weight for trajectories $x_0\to B_{\eps_0-\delta}(y_0)\to o$ is exponentially smaller than the one for trajectories $x_0\to B_{\eps_0}(y_0)^{\rm c}\to o$ for any $\varepsilon_0 > \delta>0$. 

To do this, use the triangle inequality and the monotonicity of $\ell$ to see that
$$
f_{x_0,y_0}(\eps)\leq \widetilde f_{x_0,y_0}(\eps):=\frac{\ell(\eps)}{\ell (|x_0-y_0|+\eps)}+\frac{\ell(|y_0|)}{\ell(|y_0|+\eps)}.
$$
Note that $\widetilde f_{x_0,y_0}(0)= f_{x_0,y_0}(0)$ and that 
$$
\widetilde f'_{x_0,y_0}(0)=\frac{\ell'(0)}{\ell(|x_0-y_0|)}-\frac{\ell(0)\ell'(|x_0-y_0|)}{\ell(|x_0-y_0|)^2}-\frac{\ell'(|y_0|)}{\ell(|y_0|)}.
$$
Note that for the choice $\ell(r)=(1+r)^{-\alpha}$ for some $\alpha>0$, this is negative as soon as $|x_0-y_0| (1+|x_0-y_0|)^{\alpha-1}>(1+|y_0|)^{-1}$, i.e., as soon as $y_0$ is sufficiently far away from $o$, given the distance of the transmission site $x_0$ from $y_0$. This proves the announced conclusion that a two-hop transmission from $x_0$ to the origin is strictly not optimal if the relaying hop uses a neighbourhood of $y_0$; here we used no information about the spatial relation of the three sites $x_0$, $y_0$ and $o$, but the fact that $\ell'(0)<0$. However, for the path-loss function $\ell(r)=\min\{1,r^{-\alpha}\}$, this argument does not work, since $\widetilde f'_{x_0,y_0}(0)>0$ (because $\ell'(0)=0$).

\section{Modelling discussions and conclusions}\label{sec-conclusion}

In this section, we explain our motivation for some aspects of the model and for the questions that we address. We discuss the notion of SIR, its adaptation to the high-density setting, and the effect of boundedness of the path-loss function in Section~\ref{sec-SIRdefdiscussion}. Afterwards, in Section~\ref{sec-extensions}, we comment on possible extensions of the model involving a strict SIR threshold, users sending no message or multiple messages, or time dependence. Finally, in Section~\ref{sec-GibbsVSoptimization} we provide further details and a discussion of our Gibbsian ansatz, in particular we present the version of the model where also congestion is penalized.


\subsection{The notion of SIR and its adaptation to the high-density setting }\label{sec-SIRdefdiscussion}
Note that the conventional definition of interference of a transmission from $X_i$ to $x$ is $ \sum_{j \in I^\lambda \setminus \lbrace i \rbrace} \ell(|X_j-x|)$, in contrast to our definition in \eqref{SIR}, where we added a factor of $\frac 1\lambda$, following \cite[Section 1]{HJKP15}. According to this convention, we should say ``total received power" instead of ``interference", cf. \cite[Section II.]{KB14}. As we are interested in the limit $\lambda \to \infty$, where it makes no difference whether or not we add $\frac{1}{\lambda} \ell(|X_i-x|)$ to the denominator, we stick to our notions ``SIR" and ``interference". For the same reason, our model does not include noise. However, note also our additional factor of $1/\lambda$, which we think is appropriate, at least mathematically, to our setting, in which we consider the high-density limit $\lambda\to\infty$. We actually scale the ``usual'' SIR by the density parameter. Indeed, in order to cope with an enormous number of messages in a system with 
one base station 
and a fixed bandwidth, one can either distribute the messages over a longer time stretch or decompose the messages into many smaller ones. The factor of $1/\lambda$ is a crude approximation of a combination of these two strategies. 

The assumption that the path-loss function $\ell$ is continuous at 0 comes from \cite{GT08, HJKP15} and differs from the works \cite{GK00,KB14}, which make mathematical use of the perfect scaling of the path-loss function $\ell(r)=r^{-\alpha}$, which is for this reason one of the standard choices. However, for small $r$, this is an unrealistic choice, cf.~\cite[Section I.A]{GK00}, \cite[Section I.]{GT08}. Let us also note that in case of an unbounded path-loss function, already the law of large numbers of Proposition~\ref{prop-variationbeta=0} may be wrong; indeed, it might be necessary to drop the factor of $1/\lambda$ in the denominator of our definition \eqref{SIR} of the SIR and to take $\kmax=\infty$ in order to obtain an interesting result instead. See~\cite{GK00, FDTsT07} for further details. 

\subsection{Extensions of the Gibbsian model}\label{sec-extensions}
\subsubsection{Strict SIR threshold}\label{sec-SIRthreshold} 
In a mathematical description of a telecommunication system, one typically requires that the SIR be larger than a given threshold $\tau>0$, in order that the signal can be successfully transmitted. One can also modify our Gibbs distribution in such a way that trajectories exhibiting hops with $\SIR(s^i_{l-1},s^i_l,X^\lambda)$ less than or equal to $\tau$ have probability zero, simply by changing the interference penalization value \eqref{Pinotation} to $\infty$ for such families, similarly to \cite[Section III.A]{BC12}. For $\tau$ small enough, almost surely, the modified model is well-posed for all $\lambda>0$ sufficiently large \cite[Section 5.2.3]{T18}. This means a change from the penalization function $x\mapsto \gamma/x$ (applied to $\SIR(s^i_{l-1},s^i_l,X^\lambda)$) into the function $x\mapsto \infty\times \1_{[0,\tau]}(x)$. We expect that an analogue of Proposition~\ref{prop-variationbeta=0} in \cite[Section 5]{KT17} is valid, but additional topological problems have to be addressed.

\subsubsection{Sending no or multiple messages}\label{sec-multiplemessages}

One easily sees from the proofs in \cite[Sections 2--5]{KT17} that Proposition~\ref{prop-variationbeta=0} can be extended to the situation where users send no message or multiple messages. This models the standard situation in which large messages are cut into many smaller ones, who independently find their ways through the system.

For this, we have to enlarge the trajectory probability space: to each user $X_i \in X^\lambda$, we attach the number $P_i\in\N_0$ of transmitted messages, and for each $j\in\{1,\dots,P_i\}$, there is an independent trajectory $X_i\to o$. The empirical trajectory measure $R_{\lambda,k}(\cdot)$ must be augmented by these trajectories. The main additional assumption then is that $\sum_{k=1}^{\kmax} \pi_0 R_{\lambda,k}(\cdot)$ converges to some measure $\mu_0 \in \mathcal{M}(W)$ with $0 \neq \mu_0 \ll \mu$. According to \cite[Sections 2.3.1, 5.1]{BB09}, the SIR of the transmission of one of these messages from $X_i$ to $x \in W$ should be defined as follows
\[ \SIR(X_i,~x,~((X_j,P_j))_{j \in I^\lambda})=\frac{\ell(|X_i-x|)}{\frac{1}{\lambda} \sum_{j \in I^\lambda} \ell(|X_j-x|)P_j}. \]
One could also incorporate (possibly random) sizes of the messages, which would require an additional enlargement of the trajectory space.


\subsubsection{Time-dependent versions of the Gibbsian model}\label{sec-timedependence}
We note that the notion of interference can be made more realistic according to \cite[Section I.A]{GK00} via introducing time dependence in our model. E.g., one introduces $\kmax$ discrete time slots indexed by $[\kmax]$, and for $l \in [\kmax]$, the $l$th hop of any message trajectory is assumed to happen at time $l$. Then, the interference of a transmission at time $l$ is obtained from the starting points of all hops that happen at the same time. The SIR is defined analogously to \eqref{SIR} but with this notion of interference, which depends on the entire message trajectories rather than only on the users. Time-dependent versions of our model can be set up in various ways; for example, one could allow for messages standing still or for a longer time horizon and users transmitting multiple messages over time. The new notion of SIR comes with significant changes in the behaviour of the system in the high-density limit, and we decided to defer such investigations to a later work.

\subsection{Further details and discussion of our Gibbsian ansatz}\label{sec-GibbsVSoptimization}

In Section~\ref{sec-modeldef} we introduced the special case of the model of \cite{KT17} that is most amenable for analytical investigations, i.e., the one where only interference is penalized and congestion is not, and we considered this setting in Sections~\ref{sec-largespace}--\ref{sec-largea}. 
The Gibbs distribution takes a product form \eqref{Gibbsdistribution}--\eqref{zed}, i.e., message trajectories are independent. Penalizing congestion introduces interaction between different message trajectories, and some interesting properties of our model hold only in presence of this term. 

In Section~\ref{sec-alsocongestion}, we introduce the Gibbsian model in case congestion is also penalized. Then, we explain the form of our interference penalization (i.e., the use of $1/\SIR$) in Section~\ref{sec-SIRpenaltydiscussion}. In Section~\ref{sec-varformula} we remind on the characterization \cite[Proposition 1.5, part (3)]{KT17} of the limiting trajectory measures \eqref{nukminimizerbeta=0}--\eqref{Adefnew} as the unique minimizer of a characteristic variational formula.

If congestion is also penalized, the dependencies between different trajectories make the numerical simulation of the Gibbs distribution a different task. In Section~\ref{sec-MCMC} we suggest an approximate solution for this problem using stochastic algorithms, which relates the Gibbs distribution to the optimization of the sum of the interference term and the congestion term, providing additional motivation to our model. Later, in Section~\ref{sec-gametheory}, we analyse this optimization problem in the light of traffic theory.


\subsubsection{Definition of the Gibbsian model with congestion penalization}\label{sec-alsocongestion}
Recall from Section~\ref{sec-modeldef} that elements of $\mathcal S_{\kmax}(X^\lambda)$, i.e., admissible message trajectory configurations, are denoted as $s=(s^i)_{i \in I^\lambda}$. For $i \in I^\lambda$ and $l \in \{ 0,1,\ldots,|s^i| \}$, we write $s^i_l$ for the $l$th coordinate of the trajectory $s^i$ (in particular, the 0th coordinate equals the transmitter $X_i$ and the last coordinate equals the receiver $o$). First, we need an alternative notation for interference penalization. For $s \in \Scal_{\kmax}(X^\lambda)$, we define
\[ \mathfrak S(s) = \sum_{i \in I^\lambda} \sum_{l=1}^{|s^i|} \SIR(s^i_{l-1},s^i_l,X^\lambda)^{-1}. \numberthis\label{realSIRenergy} \]
Next, we introduce congestion. For $s \in \Scal_{\kmax}(X^\lambda)$ we put
\begin{equation}
m_i(s)=\sum_{j \in I^\lambda} \sum_{l=1}^{|s^j|-1} \mathds 1 \lbrace s^j_l=X_i \rbrace, \quad i\in I^\lambda,
\end{equation}
as the number of times user $X_i$ is used as a relay in the trajectory collection $s$, and we define
\[ \mathfrak M(s) = \sum_{i \in I^\lambda} m_i(s)(m_i(s)-1). \numberthis\label{realMenergy} \]
Note that $m_i(s)(m_i(s)-1)$ is the number of ordered pairs of hops arriving at the relay $X_i$, and $m_i(s)(m_i(s)-1)=0$  if $m_i(s) \in \{ 0,1 \}$, i.e., $\mathfrak M(\cdot)$ only counts pairs of hops arriving at the same relay. (We note that in a time dependent setting, the correct analogue of $\mathfrak M(\cdot)$ would only count pairs of hops arriving \emph{simultaneously} at the same relay.)

Now, we fix $\gamma>0$ and $\beta \geq 0$, and for any $s=(s^i)_{i \in I^\lambda}\in \mathcal{S}_{\kmax}(X^\lambda)$ we put 
\begin{equation}\label{realGibbsdistribution}
\mathrm P^{\gamma,\beta}_{\lambda,X^\lambda}(s) := \frac1{Z^{\gamma,\beta}_\lambda(X^\lambda)}
\Big(\prod_{i\in I^\lambda}  \frac1{N(\lambda)^{|s^i|-1}}\Big) \exp\Big\{-\gamma \mathfrak S(s)-\beta \mathfrak M(s) \Big\}.
\end{equation}
This is the Gibbs distribution with a uniform and independent \emph{a priori} measure (see~\cite[Section 1.2.2]{KT17} for details), subject to an exponential weight with the interference term in \eqref{realSIRenergy} and the congestion term in \eqref{realMenergy}. Here
\[ Z^{\gamma,\beta}_\lambda(X^\lambda)=\sum_{r \in \mathcal{S}_{\kmax}(X^\lambda)}  \Big(\prod_{i\in I^\lambda}\frac1{N(\lambda)^{ |r^i|-1}}\Big) \exp\Big\{-\gamma \mathfrak S(r)-\beta \mathfrak M(r)\Big\} \numberthis\label{realzed} \]
is the normalizing constant, also referred to as \emph{partition function}. Note that $\mathrm P^{\gamma,\beta}_{\lambda, X^\lambda}(\cdot)$ is random and defined conditional on $X^\lambda$, and it is a probability measure on $\mathcal{S}_{\kmax}(X^\lambda)$. Note further that for $\beta=0$, i.e., in case congestion is not penalized, $\mathrm P^{\gamma,0}_{\lambda,X^\lambda}$ and $Z^{\gamma,0}_\lambda(X^\lambda)$ defined in \eqref{realGibbsdistribution}--\eqref{realzed} equal $\mathrm P^{\gamma}_{\lambda,X^\lambda}$ respectively $Z^{\gamma}_\lambda(X^\lambda)$ defined in \eqref{Gibbsdistribution}--\eqref{zed}.

\subsubsection{The interference term}\label{sec-SIRpenaltydiscussion}
The interference term $\mathfrak S(s)$ in \eqref{realSIRenergy} (see also \eqref{Pinotation}--\eqref{Gibbsdistribution}) quantifies the quality of the transmission of the messages in case they use the trajectories $s^i$ from $X_i$ to $o$. 
The choice of the {\em reciprocals} of the SIRs comes from the fact that the \emph{bandwidth} used for a transmission is defined \cite{SPW07} as 
\[ \frac{\varrho}{\log_2 (1+ \mathrm{SIR}(\cdot))}, 
 \numberthis\label{bandwidth}\]
where $\varrho$ is the data transmission rate, and $\mathrm{SIR}$ is defined as in \eqref{SIR} without the factor of $1/\lambda$ in the denominator of \eqref{SIR}. This quantity is of order $1/\lambda$ for $\lambda$ large, under the assumption that $L_\lambda \Rightarrow \mu$. In the high-density setting $\lambda\to\infty$ that we study, \eqref{bandwidth} can be approached well by (a constant times) the reciprocals of the SIR, since $\log(1+x) \sim x$ as $x \to 0$.
\cite[Section 3]{SPW07} suggests that in case of multihop communication, the used bandwidth equals the sum of the used bandwidth values corresponding to the individual hops, which explains our choice of the sum over $l$ in \eqref{realSIRenergy}. We note that the idea of using a sum of reciprocals of SIR values as a cost function to be minimized appeared also in \cite{BC12}.

\subsubsection{The limiting trajectory distribution as the minimizer of a variational formula}\label{sec-varformula}

We only explain the case where congestion is not penalized, since this corresponds to Proposition~\ref{prop-variationbeta=0}. The analogue of this result with congestion is formulated in \cite[Theorem 1.2, Proposition 1.3]{KT17}.

Note that the interference penalty term \eqref{realSIRenergy} can be expressed in terms of $(R_{\lambda,k}(s))_{k\in [\kmax]}$ as follows
\[ \mathfrak S(s) = \sum_{k=1}^{\kmax}\int_{W} R_{\lambda,k}(s)(\d x_0,\ldots,\d x_{k-1})\sum_{l=1}^k \frac{\int_{W} \ell(|y-x_l|) L_\lambda(\d y)}{\ell(|x_{l-1}-x_l|)}, \quad x_k=o. \numberthis\label{SIRenergyexpressed} \]
Recall that $S$ denotes a random variable with distribution $\mathrm P^{\gamma,0}_{\lambda,X^\lambda}=\mathrm P^{\gamma}_{\lambda,X^\lambda}$ and that subsequential limits of $(R_{\lambda,k}(S))_{k \in [\kmax]}$ as $\lambda \to \infty$ are families of measures $\Sigma=(\nu_k)_{k \in [\kmax]}$ with $\nu_k \in \Mcal(W^k)$ satisfying \eqref{i-beta=0}. For such $\Sigma$ we define an analogue of \eqref{SIRenergyexpressed} as follows
\[ \mathrm S(\Sigma) =  \sum_{k=1}^{\kmax}\int_{W}\nu_k(\d x_0,\ldots,\d x_{k-1})\sum_{l=1}^k  g(x_{l-1},x_l), \quad x_k=o. \]
Moreover, we define the following \emph{entropy term} that describes counting complexity~\cite[Sections~2.2,~3.1]{KT17}:
\[ \mathrm J(\Sigma) = \sum_{k=1}^{\kmax} \int_{W^k} \d \nu_k \log \frac{\d \nu_k}{\d \mu^{\otimes k}} + \log \mu(W) \sum_{k=1}^{\kmax} (k-1)\nu_k(W)  \in [0,\infty], \numberthis\label{Jentropy} \]
with the convention that $0\log 0=0 \log (0/0)=0$ and $\mathrm J(\Sigma)=\infty $ whenever $\nu_k$ is not absolutely continuous with respect to $\mu^{\otimes k}$ for some $k$.
Now, by \cite[Proposition 1.5, part (3)]{KT17}, for $\kmax>1$, the coordinatewise weak limit \eqref{nukminimizerbeta=0}--\eqref{Adefnew} of $(R_{\lambda,k}(S))_{k \in [\kmax]}$ equals the unique minimizer of the variational formula
\[ \inf_{\Sigma=(\nu_k)_{k=1}^{\kmax} \colon \sum_{k=1}^{\kmax} \pi_0 \nu_k=\mu} \Big(\mathrm J(\Sigma) + \gamma \mathrm S(\Sigma)\Big). \numberthis\label{beta0variation} \]
This variational formula indeed has the form ``minimize the sum of entropy and energy among all admissible trajectory families'', as announced in Section~\ref{sec-mainfeatures}. The case $\kmax=1$ is again trivial: we have already seen that in this case $(R_{\lambda,k}(S))_{k \in [1]}$ converges to $(\mu)$, which is in fact the only admissible collection of measures for the variational formula \eqref{beta0variation}, and hence also the unique minimizer.

\subsubsection{Relation to an optimization problem via Monte Carlo Markov chains}\label{sec-MCMC} 

In the light of the motivation for the exponential form of the trajectory distribution and for the two penalty terms (cf.~Sections~\ref{sec-mainfeatures}, \ref{sec-SIRpenaltydiscussion}, and \ref{sec-varformula}), it is certainly interesting to minimize the cost function $s \mapsto \gamma \mathfrak S(s) + \beta \mathfrak M(s)$ for fixed $\beta,\gamma\in(0,\infty)$. Computationally, this is in general a hard problem for high densities $\lambda$ because the cardinality of $\mathcal S_{\kmax}(X^\lambda)$ increases super-exponentially in $N(\lambda)$, and $N(\lambda)$ is of linear order in $\lambda$. Thus, computing all values of $s \mapsto \gamma \mathfrak S(s) + \beta \mathfrak M(s)$  and then extracting the maximum is only feasible for small $\lambda$.

Now, our Gibbsian trajectory distribution opens the possibility to optimize this cost function via the well-known approach of {\em simulated annealing}. Furthermore, for $\lambda$ large, it is substantially less complex to realize the Gibbs distribution using Monte Carlo Markov chains than to directly minimize the cost function.

Indeed, our Gibbs distribution favours trajectory collections with small values of the cost function. Now, let us investigate the computational complexity of the numerical realization of the Gibbs distribution $\mathrm P^{\gamma,\beta}_{\lambda,X^\lambda}$, using Monte Carlo Markov chains (see e.g.~\cite{H02}). The recent master's thesis of Morgenstern \cite{M18} investigates this question. Given the intensity $\lambda$ and the realization of the point process of users $X^\lambda$, the author finds irreducible and aperiodic Markov chains on the state space $\mathcal S_{\kmax}(X^\lambda)$, both of Gibbs sampler and Metropolis types, having the Gibbs distribution as their stationary distribution. These chains therefore converge towards $\mathrm P^{\gamma,\beta}_{\lambda,X^\lambda}$ as the number of Markovian steps tends to infinity. 

The following results have been verified in \cite{M18}:
\begin{itemize}
\item The chains can be constructed in such a way that computing their transition matrices takes only a polynomial number of operations in $N(\lambda)$.

\item For both types of chains, in the limit $N(\lambda) \asymp \lambda \to \infty$, the mixing time is at most exponential in $\lambda$. This, together with the previous observation, provides an at most exponential upper bound on the number of operations needed in order to simulate the Gibbs distribution up to a given error $\eps>0$ in total variation distance. This is certainly much more efficient than evaluating all the trajectory collections.

\item In a variant of the Gibbsian model where any user (relay) can receive at most a fixed number $m_{\max} \in \N$ of incoming hops, the mixing time is even polynomial in $\lambda$. This is a realistic modelling since, in most applications in multihop networks, each relay has a bounded capacity for receiving incoming hops (see e.g.~\cite{HJP16,HJ17}). Note that this variant is always well-defined if $\mmax \geq \kmax-1$. The results of \cite[Section 1]{KT17} also show that for large $\mmax$, this variant behaves similarly to the original model in the high-density limit. 

\item These Monte Carlo Markov chains can also be used in order to find the optimum of the cost function $s \mapsto \gamma \mathfrak S(s) + \beta \mathfrak M(s)$ for a fixed $\lambda$ and a fixed realization of $X^\lambda$, using simulated annealing. Here, one lets the transition probability of the $t$-th step of the chain depend on $t$ via replacing $(\gamma,\beta)$ by $(\gamma_t,\beta_t)$ such that $\gamma_t,\beta_t \to \infty$ sufficiently slowly as $t \to \infty$. \cite[Theorem 7.1]{M18} shows that if one chooses $\beta_t=\smfrac{\beta}{\gamma} \gamma_t \leq c_0 \log t$ for a suitably chosen $c_0=c_0(\lambda) \asymp \lambda/N(\lambda)^2$, then the Markov chain converges to the uniform distribution on the set of minimizers of the cost function. 
\end{itemize}

of convergence of such a chain to equilibrium, which is an interesting problem on its own.

\section{Game-theoretic interpretation of the optimization problem}\label{sec-gametheory}

In this section, we use the notation introduced in Section~\ref{sec-alsocongestion}. In Section~\ref{sec-MCMC} we explained how our model can be employed for obtaining a stochastic simulation algorithm for finding minimizer(s) $s$ of $\gamma \mathfrak S(s)+\beta\mathfrak M(s)$, i.e., including the congestion term. In this section, we give a more thorough discussion of this optimization problem from a game-theoretical point of view. In particular, we explain in which sense the optimum in our model is selfish or non-selfish and give a number of explicit examples for illustration. Note that in the term $\mathfrak S(s)$ there is no interaction between the trajectories (only with the users), but in the term $\mathfrak M(s)$. We therefore keep both $\beta>0$ and $\gamma>0$ fixed.



Let $X^\lambda=\lbrace X_1,\ldots,X_n \rbrace$ be fixed, where $n \in \N$. For the rest of this section, we simplify the notation as follows. We write $\mathcal S=\mathcal S_{\kmax}(X^\lambda)$ and for $i \in [n]$, $\mathcal S^i=\mathcal S^i_{\kmax}(X^\lambda)$. Let now $s =(s^i)_{i=1}^n\in \Scal$ be a collection of message trajectories. We recall that for $i \in [n]$, $|s^i|$ is the number of hops taken by the $i$th trajectory $s^i$ sent out from $X_i$ to $o$. Then, in terms of interference and congestion, the \emph{individual cost} $C_i(s)$ of $s^i$  with respect to
the entire collection $s$ is the individual interference penalization of $s^i$, together with the congestion penalization at all the relays that $s^i$ uses:
\[ 
C_i(s) = \gamma \sum_{l=1}^{|s^i|} \SIR(s^i_{l-1},s^i_l,X^\lambda)^{-1} + \beta \sum_{l=1}^{|s^i|-1} \sum_{j=1}^n (m_j(s)-1) \mathds 1 \lbrace s^i_l=X_j \rbrace. \numberthis\label{Ci} 
\]
The \emph{total cost} of the trajectory collection $s$ is defined as
\[ C(s) = \sum_{i=1}^n C_i(s)= \gamma \mathfrak S(s)+\beta \mathfrak M(s) = \gamma \sum_{i=1}^n \sum_{l=1}^{|s^i|} \SIR(s^i_{l-1},s^i_l,X^\lambda)^{-1}   + \beta \sum_{j=1}^n m_{j}(s)(m_j(s)-1). 
\] 
We say that $s$ is \emph{system-optimal} if $C(s) \leq C(s')$ for all $s' \in \Scal$. 

For a collection $s=(s^i)_{i=1}^n$ of trajectories we write $s=s_i(s^i,s^{-i})$, where $s^{-i}=(s^j)_{j \in [n], j\neq i}$.  Now, given $i$ and  $s^{-i}=(s^j)_{j\neq i}$ with $s^j \in \Scal^j$ for all $j \neq i$, a \emph{best response} of the $i$th user for $s^{-i}$ is $u^i \in \Scal^i$ such that $C_i(s_i(u^i,s^{-i})) \leq C_i(s_i(s^i,s^{-i}))=C_i(s)$ for all $s^i \in \Scal^i$. We say \cite[Section 1.3.3]{NRTV07} that $s = (s^i)_{i=1}^n$ is a \emph{pure Nash equilibrium} if $s^i$ is a best response for $s^{-i}=(s^j)_{j \neq i}$ for all $i \in [n]$.

\begin{claim}\label{claim-Nashexists}
For $\beta,\gamma,\lambda>0$, given $X^\lambda$, a pure Nash equilibrium always exists.
\end{claim}

\begin{proof}
The claim follows from the well-known result \cite[Theorem 18.12]{NRTV07} that unweighted atomic congestion games always have a pure Nash equilibrium. Indeed, the cost functions $C_i$, $i \in [n]$, and $C$ 
are the individual respectively total costs in an unweighted atomic congestion game (atomic instance) \cite[Section 18]{NRTV07}, which is defined as follows. For each $i \in [n]$, the set of all possible paths $s^i \in \mathcal S^i$ of length at most $\kmax$ from $X_i$ to $o$ via users in $X_j \in X^\lambda$ without visiting the same $X_j$ twice can be seen as the set of the strategies of the $i$th user (player) $X_i$. Indeed, for the sake of optimization of individual and total costs, we can neglect trajectories with loops since removing any loop from the trajectory of the $i$th user strictly decreases $C_i$ and does not increase $C_j$ for $j \neq i$, neither $C$. Since almost surely the points of $X^\lambda$ are pairwise distinct, the sets of possible strategies of different users are disjoint, in other words, the game is unweighted. Further, each user has a finite number of strategies. 

The cost function in this game is defined as follows. Each hop from $X_i$ to $X_j$ has a constant cost equal to $\gamma \SIR(X_i,X_j,X^\lambda)^{-1}$, and each used relay $X_j$ has a linear cost equal to $\beta (m_j(s)-1)$, depending on the trajectory collection $s$. This way, by \eqref{Ci}, the cost of the strategy of $X_i$ corresponding to $s \in \Scal$ equals $C_i(s)$.  Thus, the claim follows.
\end{proof}
\color{black}
Now, if there exists a system-optimal $s \in \Scal$ such that $C(s) < C(s')$ for all Nash equilibria $s'$, then we call $s$ a \emph{non-selfish optimum}, since there exists $i \in [n]$ such that $s^i$ is not the best response of the $i$th user for the remaining coordinates of the trajectory collection. Example~\ref{example-nonselfish} shows a two-dimensional example that has a non-selfish optimum, and Remark~\ref{remark-yesselfish} tells more about the relation of the individual and the total costs.

\begin{example}\label{example-nonselfish}
Let $d=2$, $\lambda=1$ and $\kmax=2$, and let $X^\lambda=X^1= \{ X_1, X_2, X_3 \}$, $\ell$ and $\gamma>0$ be chosen in the following way. $X_1, X_2,X_3$ and $o,X_2,X_3$ are vertices of two equilateral triangles with $X_1$ being in the interior of the latter triangle, so that $|X_1-X_2|=|X_1-X_3|$ and $|X_2|=|X_3|$, so that $\gamma \SIR(X_1,o,X^1)^{-1}=\gamma \SIR(X_i,X_1,X^1)^{-1}=1$ and $\gamma \SIR(X_i,o,X^1)^{-1}=1+q$ for all $i \in \{2,3\}$ for some $q>0$ (see Figure~\ref{figure-nonselfish}). 

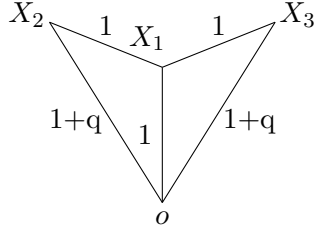
\begin{figure}
\begin{tikzpicture}
\draw (0,-0.9) -- (0,0.9) node[pos=0.5, left] {1};
\draw (0,0.9) -- (1.5,1.5) node[pos=0.5,above] {1};
\draw (0,-0.9) -- (1.5,1.5) node[pos=0.45,right] {1+q};
\draw (0,-0.9) -- (-1.5,1.5) node[pos=0.45,left] {1+q};
\draw (0,0.9) -- (-1.5,1.5) node[pos=0.5,above] {1};
\draw (0,-1.1) node {$o$};
\draw (-0.2,1.29) node {$X_1$};
\draw (1.8,1.6) node {$X_3$};
\draw (-1.8,1.6) node {$X_2$};
\end{tikzpicture}\caption{Interference weights per hop in Example~\ref{example-nonselfish}. In the relevant cases, the congestion at $X_1$ is $\beta y(y-1)$, where $y$ is the number of elements of $\{ X_2,X_3 \}$ relaying through $X_1$. }\label{figure-nonselfish}
\end{figure}


The boundedness of $\ell(|\cdot-\cdot|)$ away from 0 on $W \times W$ implies that for any $\beta>0$ and $i \in \{ 2,3 \}$, any $s^i \in \Scal^i$ that uses some $X_j$ with $j \in \{ 2,3 \}$ as a relay is suboptimal both with respect to
total and individual costs. Indeed, leaving out this relay and moving on to the next hop of the same trajectory instead decreases $C_i(s)$ without increasing any $C_m(s)$, $m\neq i$. Using analogous arguments, one easily concludes that in any system-optimal trajectory and also in any Nash equilibrium, $X_1$ submits directly to $o$, and the two users $X_2,X_3$ use either the direct link to $o$ or the two-hop path via $X_1$ to $o$. Table~\ref{table-nonselfish} shows the individual costs and the total cost in some standard representatives of these cases.

\begin{table}
\begin{tabular}{|c|c|l|l|l|} \hline
Number of hops of $s^2$ & Number of hops of $s^3$ &  $C_2(s)$ & $C_3(s)$  & $C(s)$ \\ \hline 
1 & 1  & $2+q$ & $2+q$  & $5+2q$ \\ \hline
2 & 1  & 2 & $2+q$ & $5+q$ \\ \hline
2 & 2  & $2+\beta$ & $2+\beta$ & $5+2\beta$ \\ \hline
\end{tabular}
\vspace{6pt}
\caption{Individual and total costs in standard representatives of the relevant cases in Example~\ref{example-nonselfish}.}\label{table-nonselfish}
\end{table}

The positive parameters $q$ and $\beta$ can be chosen such that the following holds. Given that $X_2$ uses its two-hop path $X_2 \to X_1 \to o$, the best response of $X_3$ is to also use its two-hop path $X_3 \to X_1 \to o$ and vice versa, so that both users using their two-hop paths forms the unique Nash equilibrium, but the system optima are the trajectory collections in which only one of them relays via $X_1$ and the other one submits directly to $o$. According to Table~\ref{table-nonselfish}, this holds if $q>0$ and $\beta \in (q/2, q)$. Thus, in such cases, the optimum is non-selfish.

Similar effects occur in all dimensions $d \geq 2$, with $d+1$ users $X_1,X_2,\ldots,X_{d+1}$ situated so that $|X_j-X_1|=|X_i-X_1|$ and $|X_1|<|X_i|=|X_j|$ for all $i,j \geq 2$. In such cases, one can choose the parameters in such a way that for all $j \geq 2$, knowing that $X_1$ transmits directly to $o$ and each $X_i$, $j \neq i \geq 2$ relays through $X_1$, the best response of $X_j$ is to use also the relayed link via $X_1$, but with respect to total costs it would be better if $X_j$ transmitted directly to $o$. Note that if this holds, it may still happen that neither of these two joint strategies is system-optimal.
\ExampleEnde
\end{example}

\begin{remark}
In the setting of our Gibbsian model, Nash equilibria are not necessarily unique. Consider Example~\ref{example-nonselfish} in the boundary case $\beta=q$. Then one easily checks that the system exhibits three different Nash equilibria, namely the three ones that appear in Table~\ref{table-nonselfish}. Also for $\beta>q$, there are two Nash equilibria, namely the ones where exactly one of $s^2,s^3$ transmits directly to $o$ and the other one via $X_1$, by the symmetry between $X_2$ and $X_3$.
\end{remark}

\begin{remark}\label{remark-yesselfish}
Now we show that in general, if plugging in an additional relay to a trajectory decreases the total cost, then it also decreases the individual cost of the transmitter of that trajectory. Thus, in particular, a situation opposite to Example~\ref{example-nonselfish} is not possible, that is, for any choice of $\ell,\beta,\gamma$, it cannot happen that $C(s)$ is a non-selfish system optimum while $C_2(s)$ or $C_3(s)$ is a Nash equilibrium.   

Indeed, consider Figure~\ref{figure-nonselfish2} with $\lambda>0$, $X_i,X_h \in X^\lambda$ and $x \in X^\lambda \cup \{ o \}$, where the direct hop from $X_i$ to $x$ has interference penalization $p_0>0$, while the two-hop path via $X_h$ has interference penalization $p_1+p_2$ with $p_1,p_2>0$. Now, if $s^{-i}=(s^j)_{j \neq i}$ is given and the number of incoming hops at $X_h$ coming from all trajectories but the one of $X_i$ equals $m \geq 0$, then the direct link from $X_i$ to $x$ has individual cost 
$p_0+K$ and the $X_i\to X_h \to o$ relayed link has individual cost 
$m+p_1+p_2+K$, for some $K \geq 0$. On the other hand, the total cost of the collection with the $X_i\to x$ direct link is 
$2m+p_1+p_2+K'$, and the one with the $X_i \to X_h \to x$ relayed link is 
$p_0+K'$, for some $K' \geq 0$. Thus, if plugging in the relay $X_h$ increases the individual cost $C_i$, then it also increases the total cost $C$. This implies the claim.
\end{remark}
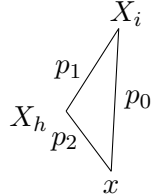
\begin{figure}
\begin{tikzpicture}
\draw (0,0.1) -- (-0.6,0.9) node[pos=0.5, left] {$p_2$};
\draw (-0.6,0.9) -- (0.1,2) node[pos=0.5,left] {$p_1$};
\draw (0,0.1) -- (0.1,2) node[pos=0.5,right] {$p_0$};
\draw (0,-0.1) node {$x$};
\draw (-1.1,0.8) node {$X_h$};
\draw (0.2,2.2) node {$X_i$};
\end{tikzpicture}\caption{A situation opposite to Example~\ref{example-nonselfish} is not possible.}\label{figure-nonselfish2}
\end{figure}

\section{Numerical studies} \label{sec-simulations}

In this section, we give numerical illustrations of various properties of the minimizer $(\nu_k)_{k\in[\kmax]}$ \eqref{nukminimizerbeta=0} of \eqref{beta0variation}, which describes the limiting empirical trajectory measure according to Proposition~\ref{prop-variationbeta=0}. We consider $\kmax=2$, $d=1,2$, $\ell$ satisfying $\ell(r) \sim r^{-4}$ as $r \to \infty$, $W$ being a ball sufficiently large such that both direct communication and two-hop communication are non-negligible, and $\mu$ being the Lebesgue measure on $W$. We do not consider congestion, i.e., we put $\beta=0$. 
We will look at the areas where one-hop and two-hop communication dominate, respectively, and the approximation of a straight line of the latter trajectories for $d=2$. We will see that the effects that we proved in Section~\ref{sec-gammatoinfty} in the limit $\gamma\to\infty$ are already very pronounced for $\gamma=1$. 

First, let us choose $\ell(r)=\min\{1,r^{-4}\}$. Let $W=\overline{B_\varrho(o)}$ be a ball around the origin $o$. We will pick $\varrho$ sufficiently large so that the effect of the path-loss function $\ell$ is strong enough in the sense that we can study areas in $W$ from which a direct hop to $o$ is preferred and areas from which a two-hop trajectory is preferred. We are interested in seeing how sharp the transition between these two areas is. By rotational invariance, we expect that the first area is a centred ball and the second the complement of a ball in $W$. Hence, we do not lose much when going to $d=1$. Using the arguments of the first paragraph of Section~\ref{sec-gammatoinftystatement}, for large $\gamma$, we expect the transition close to the point where the interference term gives the transition from optimality of one-hop trajectories to two-hop trajectories, i.e., at the radius $|x_0|$, where the number
\[ 
g(x_0,o) - \min_{x_1 \in W} \big(g(x_0,x_1) + g(x_1, o)\big) \numberthis\label{1or2} 
\]
switches the sign. Let $r_0^*$ denote that point. Our main question is whether already for moderate values of $\gamma$, we see a pronounced transition in the measures $\nu_1(\d x_0)$ and $\pi_0\nu_2(\d x_0)$ of the form that $\nu_1(\d x_0) \approx \mu(\d x_0)$ for all $x_0$ with $|x_0|<r_0^*$ and $\pi_0 \nu_2(\d x_0) \approx \mu(\d x_0)$ for all $x_0$ with $|x_0|>r_0^*$, with a fast change around $r_0^*$. Observe that both $\nu_1$ and $\pi_0 \nu_2$ have densities that are positive throughout the interior of the support of $\mu$.

In the following one-dimensional numerical example, the answer is yes, already for $\gamma=1$. The plots presented here were created using Wolfram \emph{Mathematica}.
\begin{example}\label{example-1dim}
Let $\kmax=2$, $d=1$, $W=[-5,5]=\overline{B_5(o)} \subset \R$, $\mu=\Leb|_W$ and  $\ell(r)=\min \{1, r^{-4}\}$.  According to Proposition~\ref{prop-variationbeta=0}, the minimizing measures $\Sigma=(\nu_1,\nu_2)$ are given as follows. With
\[ 
\frac{1}{A(x_0)} = \exp\Big( -\gamma \frac{\int_{-5}^5 \ell(|y|) \d y}{\ell(|x_0|)} \Big) + \frac{1}{10} \int_{-5}^5 \d x_1 \exp \Big( -\gamma \Big( \frac{\int_{-5}^5 \ell(|y-x_1|) \d y}{\ell(|x_0-x_1|)}+\frac{\int_{-5}^5 \ell(|y|) \d y}{\ell(|x_1|)} \Big) \Big), \numberthis\label{A1dim} 
\]
we have
\[ 
\nu_1(\d x_0) = \d x_0 A(x_0) \exp\Big( -\gamma \frac{\int_{-5}^5 \ell(|y|) \d y}{\ell(|x_0|)} \Big) \numberthis\label{nu11dim} \]
and
\[ \nu_2(\d x_0,\d x_1) = \frac{1}{10} \d x_0 \d x_1 A(x_0) \exp \Big( -\gamma \Big( \frac{\int_{-5}^5 \ell(|y-x_1|) \d y}{\ell(|x_0-x_1|)}+\frac{\int_{-5}^5 \ell(|y|) \d y}{\ell(|x_1|)} \Big)\Big).  \numberthis\label{nu21dim} \]
All integrals are numerically tractable for $\gamma \in[0,1]$. As seen in Figure~\ref{figure-nu1density1d}, already for $\gamma=1$, the density of $\nu_1$ is very close to the step function with a jump at the point $r_0^*$ where \eqref{1or2} switches its sign.  
\begin{figure}
\includegraphics[scale=0.68]{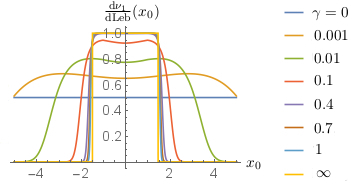}
\caption{The graphs of $x_0\mapsto\smfrac{\d \nu_1}{\d \Leb} (x_0)$ as in Example~\ref{example-1dim} for $\gamma=0, 0.001, 0.01, 0.1, 0.4, 0.7, 1, \infty$.}\label{figure-nu1density1d}
\end{figure}
Also the density of two-hops paths, $(x_0,x_1) \mapsto \smfrac{\d \nu_2}{\d \Leb^{\otimes 2}} (x_0,x_1)$, is extremely small for $|x_0-x_1|$ large, already for $\gamma=1$, so that we prefer to plot it on a logarithmic scale, see Figure~\ref{figure-nu2density1d}. 

\begin{figure}
\includegraphics[scale=0.5]{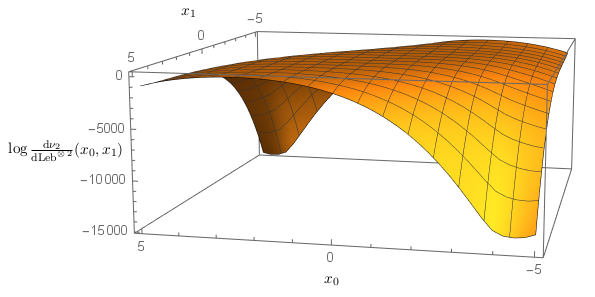}
\includegraphics[scale=0.5]{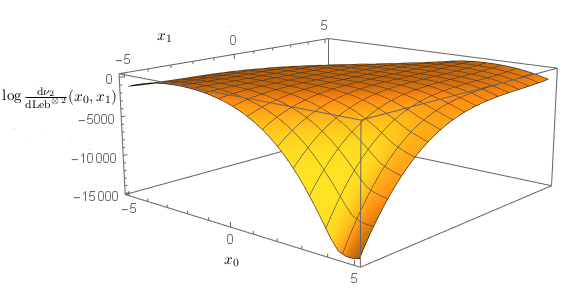}
\caption{The graphs of $(x_0,x_1) \mapsto \log\smfrac{\d \nu_2}{\d \Leb^{\otimes 2}} (x_0,x_1)$ as in Example~\ref{example-1dim} for $\gamma=1$ from two different views.}\label{figure-nu2density1d}
\end{figure}

It is not easy to read from Figure~\ref{figure-nu2density1d} what value(s) maximize(s) $\smfrac{\d \nu_2}{\d \Leb^{\otimes 2}} (x_0,x_1)$ over $x_1 \in W$ for given $x_0 \in W$. Such a maximizer can be interpreted as an \emph{optimal relay} of the transmitter $x_0$. Numerical simulations of the optimal relays turn out to be very noisy. Nevertheless, we see that this maximization problem has an interesting property. Since the normalized intensity measure $\mu/\mu(W)$ is the uniform distribution on $W$, for any $\gamma>0$ and $x_0 \in W$ fixed, the set of optimal relays of $x_0$ equals the set of minimizers of
\[ x_1 \mapsto g(x_0,x_1)+g(x_1,o) = \frac{\int_{W} \ell(|z-x_1|) \d z}{\ell(|x_0-x_1|)} + \frac{\int_{W} \ell(|z|)\d z}{\ell(|x_1|)}.\numberthis\label{2g}\]
The interpretation of this phenomenon is that the optimal value with respect to interference penalization corresponds to the largest value of $x_1 \mapsto \smfrac{\d \nu_2}{\d \Leb^{\otimes 2}}(x_0,x_1)$ for any $x_0 \in W$. Now, the second term in \eqref{2g} is optimal if $|x_1| \leq 1$. Further, since $x_1 \mapsto \int_{W} \ell(|z-x_1|) \d z$ is strictly monotone decreasing in $|x_1|$, one easily sees that if $|x_0| \leq 1$, then $x_1 = \mathrm{sgn}~ x_0$ optimizes the first term over $\{ |x_1| \leq 1 \}$. This implies that if $|x_0| < 1$, then $\smfrac{\d \nu_2}{\d \Leb^{\otimes 2}} (x_0,x_1)$ is strictly larger for $x_1=\mathrm{sgn}~ x_0$ than for any $x_1$ with $|x_1|<1$. That is, all optimal relays of $x_0$ are further away from $o$ than $x_0$ itself, and the first hop of an optimal message trajectory increases the distance from $o$! Nevertheless, we recall from Figure~\ref{figure-nu1density1d} that for $|x_0| \leq 1$, one-hop communication dominates two-hop communication already for $\gamma=1$, i.e.,
 the density of $\nu_2$ at an optimal two-hop path is still very low.

Since $x_1 \mapsto \int_{W} \ell(|z-x_1|) \d z$ varies much slower than $x_1 \mapsto \ell(|x_1|)$ for $|x_1|>1$, we expect according to \eqref{2g} that in general for $|x_0| \leq 2$, all optimal relays of $x_0$ must be situated very close to 1. In contrast, for $|x_0|$ significantly larger than 2, the optimal relays of $x_0$ are all about $x_0/2$. Indeed, in this regime, the optimization of $ x_1 \mapsto \smfrac{\d \nu_2}{\d \Leb^{\otimes 2}} (x_0,x_1)$ can be restricted to the set of $x_1$ such that both $|x_0-x_1|=|x_0|-|x_1|>1$ and $|x_1|>1$ hold. On this set, the following three properties are satisfied: $x_1 \mapsto \int_{W} \ell(|z-x_1|) \d z$ does not vary much (at $x_1=o$, its value is around 2.6613, while for $|x_1|=4$, it is about 2.3328), $\ell$ is taken at values $|x_1|$ and $|x_0-x_1|$ where it is strictly monotone decreasing, and $1/\ell$ is convex. In the context of Example~\ref{example-2dim}, where these properties hold everywhere (i.e., $\ell$ has no constant part), we will argue why they 
imply that the optimal relay of $x_0$ is close to $x_0/2$ for all $x_0 \in W$.


Our numerical results suggest that \eqref{twohop} does not hold in this example for $\Delta=W$, since otherwise, by Proposition~\ref{prop-Ma(W)tozero}, $\pi_0 \nu_2$ would be close to the zero measure on the entire communication area $W$ for large $\gamma$. The same is true for Example~\ref{example-2dim}.
\ExampleEnde
\end{example}

In the next, two-dimensional, example, we analyse the concentration of the measure $\nu_2$ \eqref{nukminimizerbeta=0} on the straight line between transmitter and receiver in the setting of Section~\ref{sec-gammatoinfty}, i.e., in case of a rotationally invariant intensity $\mu$ and a strictly monotone decreasing path-loss function $\ell$.
\begin{example}\label{example-2dim}
We choose $d=2$, $\kmax=2$, $W=\overline{B_7(o)} \subset \R^2$, $\mu=\Leb|_W$ and $\ell(r)=(1+r)^{-4}$. We note that $\ell$ is strictly monotone decreasing. Now, the one-hop trajectory measure $\nu_1$ is a measure on $W \subset \R^2$ and the two-hop one $\nu_2$ is a measure on $W^2 \subset \R^4$; they are defined as in \eqref{nukminimizerbeta=0}, analogously to the concrete case \eqref{A1dim}--\eqref{nu21dim}, with a suitable adaptation to the new parameters. 

Now, Proposition~\ref{prop-gammatoinfty}~\eqref{first-gammatoinfty} implies that for any $x_0\in W$, all maximizers of the function $x_1 \mapsto \smfrac{\d \nu_2}{\d \Leb^{\otimes 2}}(x_0,x_1)$ are situated along the straight line segment $[x_0,o]$. The higher the value of $\gamma$ is, the stronger the density $\smfrac{\d \nu_2}{\d \Leb^{\otimes 2}}(x_0,\cdot)$ concentrates around the set of maximizers, according to part~\eqref{second-gammatoinfty} of the same proposition. As in Example~\ref{example-1dim}, since $\mu/\mu(W)$ is the uniform distribution on $W$, the maximum of $x_1 \mapsto \smfrac{\d \nu_2}{\d \Leb^{\otimes 2}}(x_0,x_1)$ is taken at the minimizer(s) of $x_1 \mapsto g(x_0,x_1)+g(x_1,o)$ for any $\gamma>0$. 

Since these minimizers lie on $[x_0,o]$, we now  numerically search for an approximate minimizer of the form $x_1=\alpha x_0$ with $\alpha \in [0,1]$. It is highly expectable that the optimal $\alpha$ will be very close to 1/2 for any $x_0 \in W$. Indeed, since $1/\ell$ is convex, for $c>0$, $x_1 \mapsto c/\ell(|x_0-\alpha x_0|)+c/\ell(|\alpha x_0-o|)$ attains its minimum at $\alpha=1/2$. On the other hand, for any $x_0 \in W$, the numerator of $g(x_0,\cdot)$, i.e., $x_1 \mapsto \int_{W} \ell(|z-x_1|) \mu(\d z)$, is close to a constant on $\frac 12 W=\{ x_0/2 \colon x_0 \in W \}$. Its value at $o$ is approximately $1.0022$ and its minimal value in $\frac 12 W$ at $|x_1|=7/2$ is around $0.9798$. Therefore, $\alpha$ much smaller than 1/2 cannot be the optimal choice.
Outside $\frac 12 W$, the interference is not any more close to constant: its minimal value at $|x_1|=7$ is about $0.4770$, slightly less than the half of the value at $o$. Nevertheless, for $\alpha$ much larger than 1/2 (close to 1), the additional penalty for a longer second hop is much larger than the gain due to the larger path-loss value and better interference of the first hop. Thus, we expect that such an $\alpha$ also cannot be optimal. 

We simulate 100 points $x_0$ according to the uniform distribution $\mu/\mu(W)$ in $W$ and we compute the minimizer of $\alpha \mapsto g(x_0,\alpha x_0)+g(\alpha x_0,x_0)$ for each of these points. We observe that in all these cases, the optimal $\alpha$ is very close to 0.5: it lies strictly between 0.49 and 0.51. We note that this holds both in the regime $\{ |x_0| < r_0^* \} $ where one-hop and in the regime $\{ |x_0|>r_0^* \}$ where two-hop communication dominates for large $\gamma$. In this example, the value  of $r_0^*$ lies between 6.2989 and 6.299, so that 18 points out of 100 in the sample belong to the regime $\{|x_0|>r_0^*\}$. These results are visualized in Figure~\ref{figure-optimalrelay2d}.

\begin{figure} \!
\includegraphics[scale=0.6]{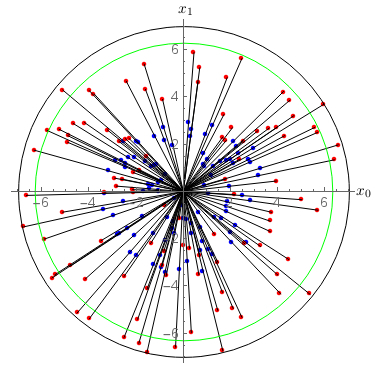}
\caption{100 uniformly distributed users $x_0$ in $W=\overline{B_7(o)}$ (in red), their approximate optimal relays $\approx x_0/2$ (blue) and the circle of radius $\approx 6.299$ separating the regimes where one-hop and two-hop communication dominates, respectively (green). The outer black circle is the boundary of $W$.}\label{figure-optimalrelay2d}
\end{figure}


%
\end{example}
\subsection*{Acknowledgements} The authors thank B. Jahnel, C. Hirsch, M. Klimm, and S. Morgenstern for interesting discussions and comments. The second author acknowledges support by the Berlin Mathematical School (BMS).

\end{document}